\documentclass[11pt]{article}
\usepackage{style}
\title{\sc{Combinatorial slicing problems of polytopes:} \\[.3em] \large \sc{How (not) to reconstruct a polytope from its slices}}

\author{Anna Birkemeyer \and Anouk E.~Brose \and Marie-Charlotte Brandenburg \and Niklas Pr\"un}
\date{}

\begin{document}

\maketitle

\begin{abstract}
    We study combinatorial aspects of hyperplane sections of polytopes, focusing on how much the combinatorics of the sections determines the combinatorics of the original polytope. We show that, in general, combinatorial information about the sections is not enough to determine even the $f$-vector of the polytope. In contrast, for sufficiently generic simple polytopes, the function recording the number of vertices of each central section determines the full combinatorial type. We organize different combinatorial slicing properties into hierarchies, separately for affine and central sections, according to the level of combinatorial structure they determine on the polytope. In analogy with classical metric slicing problems, we formulate combinatorial analogues of the Busemann-Petty problem and Bourgain's slicing problem by replacing volume with face numbers, and show that they fail in every dimension and every face dimension. 
\end{abstract}

\section{Introduction}

How much can the combinatorics of the slices of a polytope tell us about the combinatorics of the polytope itself? We study this question for both \emph{affine slices} and \emph{central slices}, i.e., intersections of polytopes with arbitrary affine hyperplanes and with hyperplanes through the origin, respectively. We investigate how face lattices of slices arise as subposets of the face lattice of the polytope, and determine relations between combinatorial properties of a polytope and those of its slices. For generic simple polytopes, we prove that already the numbers of vertices of the central slices determine their combinatorial type. Moreover, we formulate combinatorial analogues of the Busemann-Petty problem and Bourgain's slicing problem and show that they fail in every dimension.

Traditionally, slicing problems have mostly been investigated from a metric perspective. Prominent examples include descriptions of volume-extremal slices of cubes, cross-polytopes, simplices, and general norm balls \cite{ball86_cubeslicing$rn$,MeyerPajor88:lpball,Webb1996,Koenig2021,Pournin2023}. In comparison, the combinatorial study of slices of polytopes is recent. For general polytopes, \citeauthor{Khovanskii06:SectionsPolytopes} bounded the number of faces of generic slices of a simple polytope \cite{Khovanskii06:SectionsPolytopes}. More recently, \citeauthor{deloera2025numberverticeshyperplanesection} investigated the possible numbers of vertices of hyperplane sections \cite{deloera2025numberverticeshyperplanesection}. A description of the parameter space of all slices is given in \cite{Brandenburg2025}, providing an algorithmic framework for computing extremal slices and enumerating all combinatorial types. The combinatorial types of slices have attracted particular attention in the case of the cube, where affine slices have been enumerated up to dimension $6$ \cite{Nakamura1980,Fukuda1997,brandenburg2025combinatoricsslicescubes}, and central slices up to dimension $7$ \cite{brandenburg2025combinatoricsslicescubes}.

\paragraph{Combinatorics of polytopes and their slices}
Our first contributions concern the relationship between the face lattice of a slice and the face lattice of the polytope. More precisely, if a hyperplane $H$ intersects a polytope $P$, then every face of $P\cap H$ arises as the intersection of $H$ with a face of $P$. Thus, $H$ selects a collection of faces of $P$ that forms a subposet of its face lattice. We give necessary conditions for such subposets to arise from hyperplane sections (\Cref{sec:lattices}). Some of these properties are used implicitly in prior work on slices; here, we make them explicit and develop them as part of a systematic investigation.

Knowing the combinatorial types of slices provides information about many subposets of the face lattice of $P$, but not about how these subposets fit together within the whole face lattice. This leads us to the question of how much the combinatorics of the slices determine the combinatorics of the polytope itself. To make this question precise, we introduce two notions of slice-equivalence. We call $P$ and $Q$ \emph{pointwise affine-slice-equivalent \textup{(}a.s.e.\textup{)}} if, for every affine hyperplane $H$, the slices $P\cap H$ and $Q\cap H$ are combinatorially equivalent. Similarly, we call them \emph{setwise a.s.e.} if the sets of combinatorial types of all affine slices agree. We compare these notions with several combinatorial relations between $P$ and $Q$, including combinatorial equivalence ($\simeq$), equality of their oriented matroids ($\OM$), and relations between hyperplane arrangements $\mathcal H$ defined by their vertices (see \Cref{sec:hierarchies} for precise definitions). This yields hierarchies of properties for central and affine slices.
The hierarchy for central slices is given in \Cref{thm:hierarchy-central}; for affine slices, it is as follows.

\begingroup
\hypersetup{hidelinks}
    \begin{manualtheorem}{\Cref{th:affine-hierarchy}} Let $P,Q \subset \R^d$ be $d$-dimensional polytopes. The following implications and non-implications hold for every $d\ge3$, unless stated otherwise.
\begin{center}
    \centering
    $P = Q$ $\iff$
    $P,Q$ pointwise a.s.e
    $\iff$
    $\mathcal H(P \times \{1\}) = \mathcal H(Q \times \{1\})$ \\
    $\Downarrow \ \not\Uparrow$ \\
    $\OM(P) = \OM(Q)$ $\iff$
    $ \OM (\mathcal H(P \times \{1\})) = \OM(\mathcal H(Q \times \{1\}))$ \\
    $\Downarrow \ \not\Uparrow$ \hspace{15em}  $\Downarrow \ \not\Uparrow$ \\
    $P, Q$ 
    setwise a.s.e
    \hspace{4em} $\Large\substack{\Large \not \Rightarrow \text{ if } d=3\\ \hspace{-2.4em} \Large \not \Leftarrow}$ \hspace{1.5em} $P\simeq Q$ \hspace*{1em}
    \end{center}
    \end{manualtheorem}
\endgroup

The hierarchies show that the combinatorics of slices know little about the combinatorics of the polytope itself.
Thus, rather than asking which combinatorial structures are determined by the slices of polytopes in general, one can ask under which additional assumptions a polytope can be combinatorially reconstructed from partial information about its slices.
Reconstruction from partial information is a classical theme in both geometry and combinatorics. In geometric tomography, one asks to reconstruct convex or star bodies from \emph{metric} information about their sections or projections \cite{gardner-book}. Notably, every centrally symmetric convex body or star body is uniquely determined by its spherical Radon transform, i.e., by the volumes of all central slices; see \cite{boeroeczky09_stabledeterminationconvex} for generalizations. On the combinatorial side, \textcite{Blind1987} proved that simple polytopes are reconstructible from their vertex-edge graphs; see also \cite{Kalai1988,Bayer2018}. Reconstructing a graph from certain subgraphs is the subject of the famous
Kelly–Ulam Reconstruction Conjecture \cite{Bondy77_survey_graph-reconstruction,ONeil_UlamGraphReconstructionConjecture}. In a related direction, it is conjectured that affine slices of the standard cube are reconstructible from their graphs \cite[Question 1]{brandenburg2025combinatoricsslicescubes}. Here, instead, we ask whether the combinatorial type of the ambient polytope can be reconstructed from combinatorial information about its slices. We show that, for simple polytopes with sufficiently generic vertices, already the numbers of vertices of the central slices determine the graph, and hence the full combinatorial type of the polytope.

\begingroup
\hypersetup{hidelinks}
    \begin{manualtheorem}{\Cref{th:simple-polytope-reconstructible}}
        Let $P \subset \R^d$ be a simple $d$-dimensional polytope containing the origin in its interior.
        If any three vertices of $P$ are linearly independent,
        then the combinatorial type of $P$ is reconstructible from the function $\varphi_P: \Sph^{d-1} \to \N$ given by $ \varphi_P(u) = f_0(P \cap u^\perp)$. 
    \end{manualtheorem}
\endgroup

By duality between central sections and projections, \Cref{th:simple-polytope-reconstructible} equivalently reconstructs simplicial polytopes with generic facet normals: If $P$ is a simplicial $d$-dimensional polytope such that any $3$ facet normals are linearly independent, then the combinatorial type of $P$ is reconstructible from $\varphi_P: \Sph^{d-1} \to \N$, $\varphi_P(u) = f_{d-2}(\pi_u(P))$, where $\pi_u$ denotes the orthogonal projection in direction $u$.

\paragraph{Combinatorial Bourgain's slicing problem}
Since reconstruction of the combinatorial type from the combinatorics of slices fails for general polytopes, we turn to weaker comparison questions, replacing equality by inequalities. These questions are inspired by classical slicing problems from convex geometry.
The \emph{Busemann-Petty problem} compares two centrally symmetric bodies through the volumes of their central slices: if the volume of every central slice of one body is bounded by the volume of the corresponding central slice of the other body, does the same inequality hold for the volumes of the bodies themselves? 
In a series of works, the Busemann-Petty problem was shown to hold in dimensions $\leq 4$, and was disproved in every higher dimension \cite{Gardner19942,Gardner94:IntBodies,Koldobsky1998,GKS99:BusemannPetty,Zhang1999}. Recently, \textcite{guan2024notebourgainsslicingproblem} and \textcite{KlartagLehec25:AffirmativeBourgain} showed that the corresponding inequality does indeed hold up to a universal constant, resolving the almost 40-year-old \emph{Bourgain's slicing problem} (see \cite{Klartag2022} for a survey).
In analogy to these classical problems, we formulate the following combinatorial versions for polytopes. We denote by $f_k(P)$ the number of $k$-dimensional faces of a polytope $P$.

\begin{problem} \label{problems}
    Let $P,Q\subset \R^d$ be $d$-dimensional polytopes containing the origin in their interiors, let $k \in [d-1]$ and assume that $f_{k-1}(Q \cap u^\perp) \leq f_{k-1}(P \cap u^\perp)$ holds for all $u \in \Sph^{d-1}$.
    \begin{enumerate}[label=\textup{(}\alph*\textup{)}]
        \item \emph{(Combinatorial Busemann-Petty problem)} Does it follow that $f_{k}(Q) \leq f_k(P)$?\label{prob:BP}
        \item \emph{(Combinatorial Bourgain's slicing problem)} Does there exist a universal constant $C>0$ such that $f_{k}(Q) \leq C \cdot f_k(P)$?\label{prob:Bourgain}
        \item Does there exist a constant $C_d > 0$, depending on the dimension $d$, such that $f_{k}(Q) \leq C_d \cdot f_k(P)$?\label{prob:dim}
    \end{enumerate}
\end{problem}

\begingroup
\hypersetup{hidelinks}
    \begin{manualtheorem}{\Cref{th:disprove-BP}}
    All three statements in \Cref{problems} are false. In particular,
    the combinatorial Busemann-Petty problem and the combinatorial Bourgain's slicing problem are false in every dimension $d$ and for every $k \in [d-1]$.
\end{manualtheorem}
\endgroup

Although \Cref{problems} is stated for central sections through the origin, the result of \Cref{th:disprove-BP} extends to arbitrary affine sections (\Cref{rem:bpp-affine}).

\paragraph{Organization} This article is organized as follows. In \Cref{sec:lattices} we describe face lattices of slices as subposets of face lattices of $P$. We then describe transformations which preserve slices, and characterize slices of pyramids and prisms in \Cref{sec:constructions}. Afterwards, in \Cref{sec:reconstruction-central}, we consider classes of polytopes for which we can reconstruct the combinatorics of the polytope from its slices. \Cref{sec:hierarchies} concerns the hierarchies of slicing properties, and provides counterexamples for many of the relations.  Finally, in \Cref{sec:slicing-problems} we construct counterexamples to the combinatorial versions of the Busemann-Petty problem and Bourgain's slicing problem.

\paragraph{Acknowledgments} We thank Diana Wilke for many inspiring discussions.
The work on this article started in the context of the \emph{Dive into Research}, a research experience for Bachelor and Master students, supported by the SPP 2458 ``Combinatorial Synergies'', funded by the Deutsche Forschungsgemeinschaft (DFG, German Research Foundation). During the program, all authors received funding through SPP 2458. In addition, MB was supported by the SPP 2458 and AEB was partially supported by NSF grants 2348578 and 2434665.

\section{Face lattices of slices as subposets}\label{sec:lattices}

To understand which combinatorial types can occur as slices of a fixed polytope $P$, we first describe how the faces of a slice arise from faces of $P$. This yields a canonical embedding of the face lattice of a slice into the face lattice of $P$. In this section, we describe necessary properties of the subposets that can arise in this way, thus yielding candidates for combinatorial types of slices. 

Throughout this article, a polytope $P \subset \R^d$ is assumed to be $d$-dimensional, and a \emph{slice} of $P$ is the intersection of $P$ with an affine hyperplane (possibly empty).
For a polytope $P \subset \R^d$, we denote by $\verts(P)$ its set of vertices and by $\mathcal L(P)$ its face lattice. We write $G \preceq F$ for faces $G,F$ of $P$ if $G$ is a face of $F$.
For background on the combinatorics of polytopes, we refer to \cite[Ch.~1-2]{Ziegler95:LecturesPolytopes}.

Given a hyperplane $H \subset \R^d$, the combinatorial type of the slice $P \cap H$ is completely determined by the faces of $P$ which have nonempty intersection with $H$. We thus consider
$$
    \Fcal_P(H):= \{ F \in \Lcal(P) \colon F \cap H \neq \emptyset \} \cup \{ \emptyset\}\ .
$$
We say that a face $F$ is \emph{properly intersected} by $H$ if $\relint(F) \cap H \neq \emptyset$ and $F \not \subset H$. 
Any face in $\Fcal_P(H)$ is either entirely contained in $H$, properly intersected by $H$, or only intersected in its boundary. We partition these faces accordingly:
\begin{align*}
    \mathcal F^{\subset}_P(H) &= \{ F \preceq P \mid F \subset H \} \ , \\
     \mathcal F^{\circ}_P(H) &= \{ F \preceq P \mid F \not\subset H \text{ and } \relint(F) \cap H \neq \emptyset\} \ , \\
     \mathcal F^{\partial}_P(H) &= \{ F \preceq P \mid F \cap H \neq \emptyset \text{ and } \relint(F) \cap H = \emptyset\} \ .
\end{align*}

When $P$ and $H$ are fixed, we just write $\Fcal^\subset, \Fcal^\circ$ and $\Fcal^\partial$. Next, we study the structure of these sets as subsets of $\Lcal(P)$ and show that $\Fcal^\subset$ and $\Fcal^\circ$ are enough to determine $\Fcal^\partial$. Recall from \cite[Ch.~3]{stanleyEC} that a subset $S$ of a poset is a \emph{down-set} if it is downward closed, i.e., if 
\[
    x \in S, y \preceq x \implies y \in S \ .
\]
Similarly, a subset $S$ is an \emph{up-set} if it is upward closed, i.e., if 
\[
    x \in S, y \succeq x \implies y \in S \ .
\]

\begin{lemma}\label{lem:upward-downward-closure}
    Let $P$ be a polytope and $H$ be an affine hyperplane. Then $\mathcal F^{\subset}$ is a down-set of $\Lcal(P)$,  $\mathcal F^{\circ}$ is an up-set of $\Lcal(P)$ and $\mathcal F^{\partial}$ is uniquely determined by $\mathcal F^{\subset}$ and $\mathcal F^{\circ}$ as
     \[
        \Fcal^\partial = \{ F \preceq P \mid F \not \in \Fcal^\circ \cup \Fcal^\subset  \text{ and} \exists G \preceq F, G \neq \emptyset \text{ with} \ G \in \Fcal^\subset\} \ .
     \]
\end{lemma}

\begin{proof}
    If $P \cap H = \emptyset$, then $\Fcal^\subset = \{\emptyset\}$, and the other sets are empty. Assume $P \cap H \neq \emptyset$. Let $F \in \mathcal F^\subset$, i.e., $F \subset H$, and $G \preceq F$. Then, since $F \subset H$, we have $G \subset H$ and therefore $G \in \mathcal F^\subset$. Thus, $\mathcal F^\subset$ is a down-set.
    Conversely, let $F \in \mathcal F^{\circ}$ and $F \preceq G$. Since $F \cap H \neq \emptyset$ and $F \not \subseteq P \cap H$, we also have $G \cap H \neq \emptyset$ and $G \not \subseteq P \cap H$. If $\relint (G) \cap H = \emptyset$, then $H$ is a supporting hyperplane of $G$. Since $F$ is a face of $G$, $H$ must also be supporting hyperplane of $F$, which contradicts the assumption that $F \in \mathcal F ^\circ$. Therefore $\relint (G) \cap H \neq \emptyset$ and $G \in \mathcal F^\circ$. Thus, $\mathcal F^\circ$ is an up-set. 
    $\mathcal F^\partial$ can be expressed in terms of $\mathcal F^\subset$ and $\mathcal F^\circ$ as follows:
    A face $F$ of $P$ is in $\Fcal^\partial$ if and only if $H$ is a supporting hyperplane of $F$ not containing $F$ itself, that is, $F \cap H$ is some proper face $G$ of $F$, and $G \in \Fcal^\subset$. Hence $\Fcal^\partial = \{ F \preceq P \mid F \not \in \Fcal^\circ \cup \Fcal^\subset, \text{ and} \exists G \preceq F , G \neq \emptyset \text{ with} \ G \in \Fcal^\subset\}$.
\end{proof}

\begin{remark}\label{rem:minimal-elements}
    Note that \Cref{lem:upward-downward-closure} implies that $\Fcal^\subset$ is uniquely determined by its maximal elements (with respect to containment), and $\Fcal^\circ$ is uniquely determined by its minimal elements. However, it is not hard to see that $\Fcal^\subset$ is also uniquely determined by its minimal elements, that is, by the vertices in $\Fcal^\subset$.
\end{remark}
Next we want to understand how $\Lcal(P \cap H)$ embeds into $\Lcal(P)$ as a subposet. The following lemma describes faces of the slice $P \cap H$ in terms of these sets of faces of $P$.

\begin{lemma}\label{lem:types-of-faces-in-slices}
    Let $P$ be a polytope and $H$ be an affine hyperplane with $P \cap H \neq \emptyset$. The faces of the slice $P \cap H$ fall into the following categories:
    \begin{enumerate}[label=\textup{(}\roman*\textup{)},itemsep=-5pt, topsep=-5pt]
        \item $F$ for $F \in \mathcal F^{\subset}$. In this case, $\dim(F \cap H) = \dim(F)$.
        \item $F \cap H$ for $F \in \mathcal F^\circ$. In this case, $\dim(F \cap H) = \dim(F) - 1$.
    \end{enumerate}
\end{lemma}

\begin{proof}
    Let $F$ be a proper face of $P\cap H$, and let $G$ be the inclusion-minimal face of $P$ containing $F$. Since all supporting hyperplanes of $G$ are supporting for $G \cap H$, we have that $G \cap H$ is a face of $P \cap H$, and $F \subseteq G \cap H$. On the other hand, it is straightforward to show that every proper face of $P \cap H$ arises as the intersection of $P \cap H$ with a supporting hyperplane of $P$. Thus, $F = G' \cap H$ for some face $G'$ of $P$, and minimality implies $G \preceq G'$, which proves the reverse inclusion $G \cap H \subseteq G' \cap H = F$. If $G \subset H$, then $F = G$ and we have $F \in \mathcal F^\subset, \dim(F \cap H) = \dim(F)$. Otherwise, we have $F = G \cap H$ with $\dim(F) = \dim(G) - 1$, so $F \in \mathcal F^\circ$.
\end{proof}

The previous statement shows that, to every $(k-1)$-face of a slice, we can associate a unique face of $P$, which is either of dimension $k$ or $(k-1)$. 
Thus, not every $(k-1)$-face of $P\cap H$ arises as the proper intersection of $H$ with a $k$-face of $P$. Nonetheless, the following statement shows that the number of $(k-1)$-faces of the slice is bounded by the number of $k$-faces of $P$.

\begin{proposition}\label{lemma:fkmin1_of_slice_leq_fk_of_poly}
    Let $P \subset \R^d$ be a $d$-dimensional polytope, and $H \subset \R^d$ be an affine hyperplane. Then for every $k \in [d]$ we have $f_{k-1}(P \cap H) \leq f_k(P)$.
\end{proposition}

\begin{proof}
    We describe an injection $\varphi$ from the $(k-1)$-faces of $P\cap H$ to the $k$-faces of $P$. Let $F$ be a $(k-1)$-dimensional face of $P \cap H$. From \Cref{lem:types-of-faces-in-slices} we get that $F \in \mathcal F^\subset$ or $F = G \cap H$ for $G \in \mathcal F^\circ$.
    In the second case, the face $G$ is a unique $k$-dimensional face of $P$, and we define $\varphi(F) = G$.
    In the first case, $F$ is itself a $(k-1)$-dimensional face of $P$.
    We argue that there exists some $k$-dimensional face $G$ of $P$ such that $F \prec G$ and $G \not \subseteq H$, such that $G$ contains no other $(k-1)$-dimensional face contained in $\Fcal^\subset$. Then, we can choose $\varphi(F) = G$ as an injective map.
   For the existence, suppose that all such $k$-faces $G$ containing $F$ lie in $H$. Recall that the \emph{tangent cone} of a face $F$ can be written as 
   \begin{align*}
   \operatorname{tcone}(F) &= \cone \lp \{ y - x \mid x \in F, y\in P \} \rp \\
   &= \cone \lp \{ y - x \mid x \in F, y \in G, \text{ for some $k$-face $G$ containing $F$} \} \rp \ .
   \end{align*}
   The second line implies $F + \operatorname{tcone}(F) \subset H$. On the other hand, the first line implies that
   that $P \subset F + \operatorname{tcone}(F)$, which implies $P \subset H$, a contradiction.
For injectivity, suppose $G$ is a $k$-face of $P$, with $F \subset G \not \subset H$, that is, $G \cap H =F$. Note that any subset $S \subseteq G$ contained in $H$ satisfies $S \subseteq G \cap H = F$. Hence, the only $(k-1)$-face of $P$ contained in $G \cap H$ is $F$. Thus, $\varphi$ is injective and the claim follows. 
\end{proof}

Given an affine hyperplane $H$, we define the \emph{canonical embedding} of  the face lattice $\mathcal L(P \cap H)$ of the slice into $\mathcal L(P)$ as 
\begin{align*}
    \varphi_H: \mathcal L(P \cap H) &\to \mathcal L(P) \\
    G &\mapsto \text{ inclusion-minimal face containing } G.
\end{align*}
By \Cref{lem:types-of-faces-in-slices}, if $\varphi_H(G)=F$, then either $G=F$ for $F \in \Fcal^\subset$ or $G=F \cap H$ for $F \in \Fcal^\circ$. In particular, the image of this embedding is $\Fcal_P^\subset(H) \cup \Fcal_P^\circ(H)$, and we deduce the following:

\begin{proposition}\label{prop:face-lattice-of-slice}
    The face lattice $\Lcal(P \cap H)$ is the induced subposet of $\Lcal(P)$ on $\Fcal_P^\subset(H) \cup \Fcal_P^\circ(H)$.
\end{proposition}

\begin{proof}
    By \Cref{lem:types-of-faces-in-slices} there is a bijection between faces of $P \cap H$ and $\Fcal_P^\subset(H) \cup \Fcal_P^\circ(H)$. Clearly faces $G_1, G_2$ of $P \cap H$, satisfy $G_1 \subseteq G_2$ if and only if $\varphi_H(G_1) \subseteq \varphi_H(G_2)$.
\end{proof}

\begin{example}\label{ex:face-lattice-embedding}

In \Cref{fig:face-lattice-embedding} we give an example of the canonical embedding of the face lattice of a slice of the $3$-cube ($\square_3 \cap H$), into the face lattice of the cube. We draw face lattices without their maximal and minimal elements and denote faces by the string of numbers corresponding to the vertices they contain. In this example the depicted hyperplane contains the faces $\Fcal^\subset_{\square_3}(H) = \{ \emptyset, 1,2, 12 \}$, marked in blue, and properly intersects the faces $\Fcal^\circ_{\square_3}(H) = \{ 56, 67, 1458, 5678, 2367, \square_3 \}$, marked in green. Cover relations of faces in $\square_3 \cap H$ that were not cover relations of faces of the $3$-cube are drawn with dashed lines. Such a new cover relation only appears between an inclusion-maximal face in $\Fcal^\subset$ covered by an inclusion-minimal face in $\Fcal^\circ$, and their dimensions in the $3$-cube differ by two.
    \begin{figure}[ht!]
    \centering
    \begin{subfigure}{0.4\textwidth}
        \centering
        \includegraphics[width=0.7\linewidth]{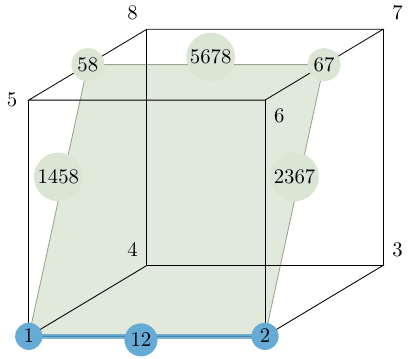}
        \caption{A square slice $\square_3 \cap H$ of the $3$-cube.}
    \end{subfigure}
    \hfill
    \begin{subfigure}{0.55\textwidth}
        \centering
        \includegraphics[width=0.6\linewidth]{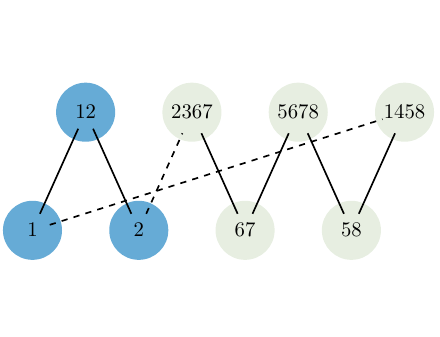}
        \caption{Face lattice of the square slice (without $\square_3 \cap H$ and $\emptyset$).}
    \end{subfigure}
    \\[2em]
    \begin{subfigure}{0.9\textwidth}
        \centering
        \includegraphics[width=.7\linewidth]{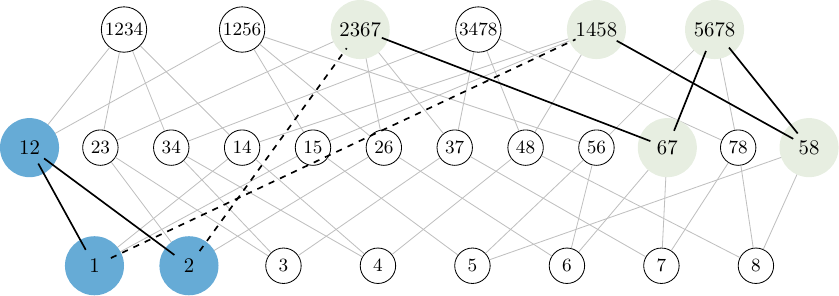}
        \caption{Embedded face lattice of the $\square_3 \cap H$ in the face lattice of the cube (without $\square_3$ and $\emptyset$).}
    \end{subfigure}

    \caption{Canonical embedding of the face lattice of a slice of the $3$-cube (\Cref{ex:face-lattice-embedding})}

    \label{fig:face-lattice-embedding}

\end{figure}
\end{example}

\goodbreak

Conversely, we are interested in understanding which subposets $\Scal$ of $\Lcal(P)$ can arise as the embedded face lattice of a slice of $P$. It is not difficult to see that combinatorially equivalent polytopes need not admit the same combinatorial types of slices (see \Cref{ex:perturbed_bipyramid}), so it is not possible to find sufficient conditions on subposets $\Scal$ of $\Lcal(P)$ such that $P$ has a slice whose face lattice is $\Scal$. Necessary conditions include that $\Scal$ will have to be the face lattice of some $(d-1)$-dimensional polytope, that $\Scal$ has a partition into an up-set and a down-set in $\mathcal{L}(P)$, and that $\Scal$ ``separates'' the graph of $P$ into at most two connected components. To see the third condition: If a hyperplane $H$ properly intersects $P$, then removing those vertices and edges from the graph of $P$ that $H$ intersects, leaves two connected subgraphs. If $H$ is a supporting hyperplane of $P$, then removing the intersected vertices and edges from the graph of $P$, leaves a (smaller) graph with one connected component.

We say that a hyperplane $H$ \emph{geometrically realizes} $\mathcal S$ if for the canonical embedding
$ \varphi_{H}(\mathcal L(P \cap H)) = \Scal$.
We identify a hyperplane $H_u(\beta) = \{x \in \R^d \mid \langle u, x \rangle + \beta =0\}$ with its vector $(u, \beta)$ of parameters. 
The space of hyperplane parameters which geometrically realize $\mathcal S$ is thus
\[
    \mathcal C(\mathcal S) = \{ (u,\beta) \in \R^{d+1} \mid \mathcal  S = \varphi_{H_u(\beta)}(\mathcal L(P \cap H_u(\beta))) \} \ . 
\]

\begin{proposition}\label{lem:realized-by-opp-cones}
    Let $P \subset \R^d$ be a polytope, and let $\mathcal S$ be a subposet of $\mathcal L(P)$. Then the space $\mathcal C(\mathcal S)$ of parameters of hyperplanes which geometrically realize $\Scal$ is a union of two opposite convex cones, i.e., $\mathcal C(\mathcal S) = C \cup (-C)$ for a (non-closed) polyhedral cone $C$.
\end{proposition}

\begin{proof}
If no hyperplane exists that geometrically realizes $\mathcal S$, then $\mathcal C(\mathcal S) = \emptyset$ is the union of two empty polyhedral cones. We thus focus on the case where at least one such hyperplane exists. In this case we know that $\Scal$ admits a partition into an up-set $\Ucal$ and down-set $\Dcal$, and furthermore if some hyperplane realizes $\Scal$, then it separates the graph of $P$ into at most two connected components, which can be determined from $\Scal$ itself:
Formally, let $G'$ be the graph on vertex set $V' = \verts(P) \setminus V(\Dcal)$ and edge set $E' = \edges(P) \setminus E(\Scal) \cup \{vw \colon v,w \in V(\Dcal)\}$, where $V(\cdot)$ and $E(\cdot)$ denote the $0$-dimensional and $1$-dimensional faces of the poset $\Scal$, respectively.
If $G'=(V',E')$ has two components, then let $V' = A\sqcup B$ be the separation of the vertices into the two components. If $G'$ has one connected component, then let $V'=A$ and $B = \emptyset$. The space of hyperplanes inducing this separation is thus the union of the cone
\begin{align*}
    C &=
    \left\{(u,\beta) \in \R^{d+1} \ \middle|\ 
    \begin{array}{ll}
    \langle u,v\rangle + \beta > 0 & \text{for all } v \in A, \text{ and}\\
    \langle u,v\rangle + \beta < 0 & \text{for all } v \in B, \text{ and}\\
    \langle u,v\rangle + \beta = 0 & \text{for all } v \in V(\mathcal D)
    \end{array}
    \right\} \\[4pt]
    &= \bigcap_{v \in A} (\sma v \\ 1 \strix ^\perp)^+ \cap \bigcap_{v \in B} (\sma v \\ 1 \strix ^\perp)^- \cap \bigcap_{v \in V(\Dcal)} \sma v \\ 1 \strix ^\perp 
\end{align*}

with its opposite $-C$.
\end{proof}

The results in this section show that every slice determines a distinguished subposet of the face lattice of $P$, namely the image of the canonical embedding $\varphi_H:\Lcal(P\cap H)\to\Lcal(P)$. Such subposets satisfy several necessary conditions: they arise from a partition into faces contained in $H$ and faces properly intersected by $H$, this partition is compatible with the up-set/down-set structure of subposet, and the corresponding hyperplane separates the graph of $P$ into at most two connected components. 

While these conditions are able to provide obstructions, they do not fully characterize which subposets arise as embedded face lattices of slices. In particular, realizability depends not only on the abstract face lattice, but also on the geometric realization of the polytope, and hence on the oriented matroid of its vertices (see \Cref{sec:hierarchies} for a definition). This motivates the following question.

\begin{question}
    Is there a characterization of subposets $\Scal$ of a face lattice $\Lcal$, such that there exists a polytope $P$ with $\Lcal(P) = \Lcal$ and a hyperplane $H$ such that $\varphi_H(\Lcal(P \cap H)) = \Scal$?
\end{question}

\section{Transformations and constructions}\label{sec:constructions}

The combinatorics of slices is highly sensitive to the polytope and the position of the hyperplane, making general statements difficult for arbitrary polytopes. In this section, we focus on settings where such statements are still possible: transformations preserving the combinatorics of slices, and the basic constructions of prisms, generalized prisms, pyramids, and partial results for bipyramids. 

Recall that if $P\subset\R^d$ is a polytope containing the origin in its relative interior, then the \emph{standard bipyramid} over $P$ is the polytope $\bipyr(P) = \conv(P\times \{0\}, \pm e_{d+1})$, and any polytope combinatorially equivalent to $\bipyr(P)$ is called a \emph{bipyramid} over $P$.
The following example illustrates the difficulties already arising when considering combinatorial bipyramids.

\begin{example}[Perturbed Bipyramids]\label{ex:perturbed_bipyramid}
    The following example shows that combinatorially equivalent polytopes need not admit the same combinatorial types of slices, not even when we demand central symmetry and only consider central slices. Let $P = \conv \{ \pm e_1, \pm e_2, \pm (e_1+e_2), \pm e_3, \}$ be the bipyramid over a hexagon, and let
    \[
    Q = \conv \{ \pm e_1,\, \pm e_2,\, \pm(e_1+e_2+\varepsilon e_3), \, \pm e_3  \} 
    \]
    for some small enough $\varepsilon  > 0$. Then all slices of $P$ have at most $8$ vertices, while $Q$ also admits central slices with $12$ vertices, see \Cref{fig:perturbed-bipyramids}. To obtain an analogous example in dimension $d > 3$, we can take iterated bipyramids over this example, as follows. If $P$ and $Q$ are combinatorially equivalent, then so are $\bipyr^{d-3}(P)$ and $\bipyr^{d-3}(Q)$, however if $P \cap u^\perp \not \simeq Q \cap u^\perp $, then we get 
\[\bipyr^{d-3}(P) \cap \sma u \\ 0 \strix ^\perp \simeq \bipyr^{d-3}(P \cap u^\perp) \not \simeq \bipyr^{d-3}(Q \cap u^\perp) \simeq \bipyr^{d-3} (Q) \cap \sma u \\ 0 \strix ^\perp.
   \qedhere \]
\end{example}
    \begin{figure}[H]
    \centering
    \hspace{2.5em}
    \begin{subfigure}[c]{0.45\textwidth}
        \includegraphics[width=0.75\textwidth]{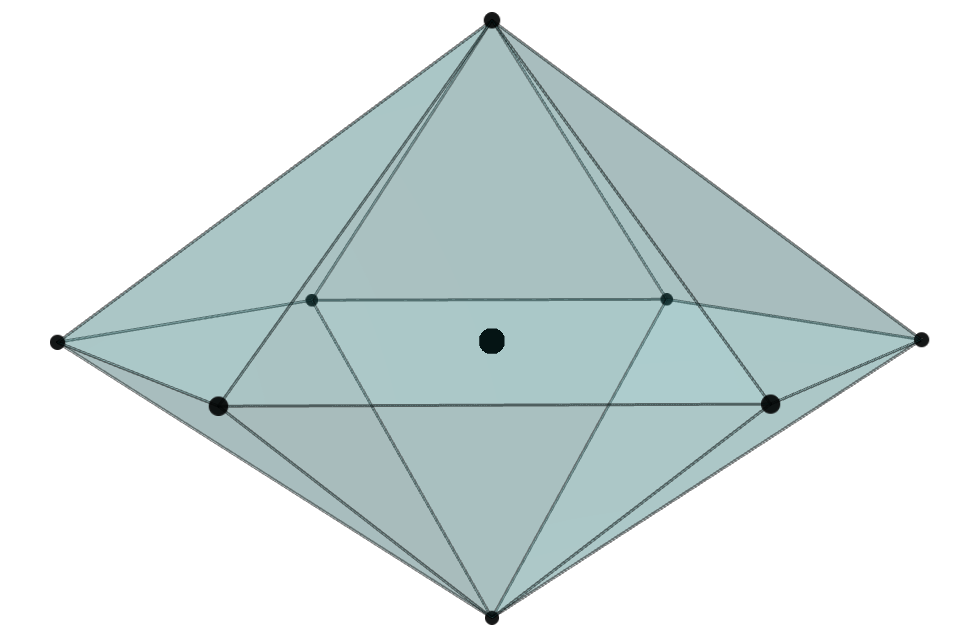}
    \end{subfigure}
    \begin{subfigure}[c]{0.45\textwidth}
         \includegraphics[width=0.75\textwidth]{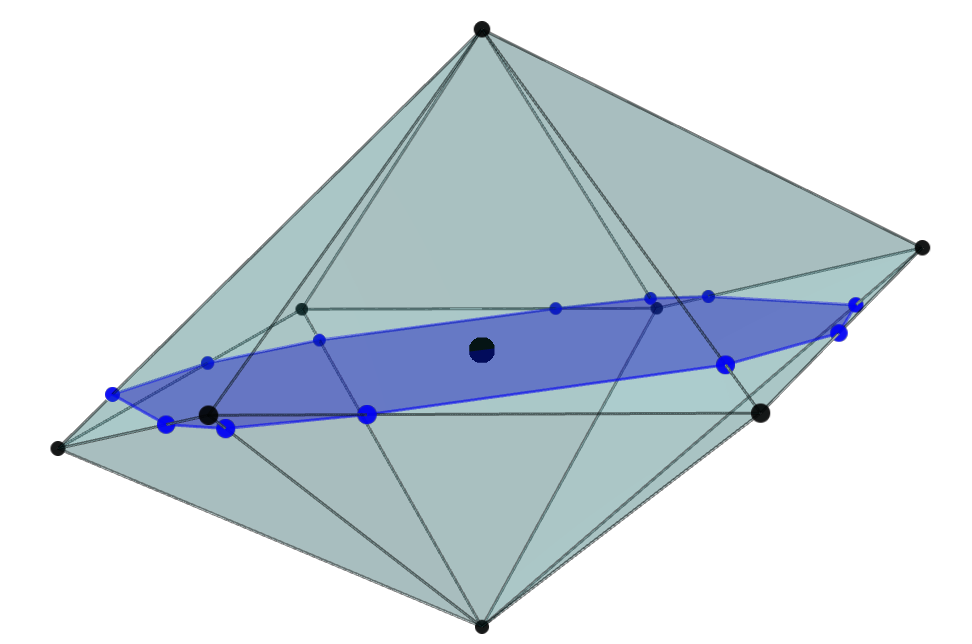}
    \end{subfigure}
    \caption{(Centrally symmetric) polytopes that are combinatorially equivalent but admit different (central) slices (\Cref{ex:perturbed_bipyramid}).}
    \label{fig:perturbed-bipyramids}
\end{figure}

We want to capture transformations of polytopes which preserve the behavior of slices.
 We say that a map $\varphi$ \emph{preserves} (central / affine) slices of a polytope $P$ if $\varphi(P)$ is a polytope, and for every (central / affine) slice $P \cap H$ of $P$ it holds that $\varphi(P \cap H)$ is a (central / affine) slice of $\varphi(P)$ that is combinatorially equivalent to $P \cap H$. We begin with a useful observation based on the fact that for injective maps we have $\varphi(P\cap H) = \varphi(P) \cap \varphi(H)$.

\begin{lemma}\label{lem:affine-transformations}
Injective affine maps preserve affine slices. Injective linear maps preserve both affine and central slices.
\end{lemma}

\begin{proof}
    For an injective map  $\varphi$ we have $\varphi(P \cap H)=\varphi(P) \cap \varphi(H)$ for any polytope $P$ and any affine hyperplane $H$. If $\varphi$ is affine, then $\varphi(H)$ is an affine hyperplane for any affine hyperplane $H$, so $\varphi(P\cap H) = \varphi(P) \cap \varphi(H)$ is a slice of the polytope $\varphi(P)$.
\end{proof}

We now show that a similar statement holds for projective transformations of polytopes. Recall that a projective transformation is a rational linear map 
$\varphi$
defined by \(A \in \R^{d \times d}\), \(b, c \in \R^d\), where

\[
\varphi(x) = \frac{A x + b}{c^\top x + \gamma},  \qquad \text{such that} \qquad
\det \begin{pmatrix} 
A & b \\
c^\top & \gamma
\end{pmatrix} \neq 0  \ .
\]

Note that $\varphi$ is only defined on $\R^d \setminus L$, where $L = H_{c}(\gamma)$, and that $\varphi$ is injective on this domain. A projective transformation is \emph{permissible} for a polytope $P$ if $c^\top x + \gamma \neq 0$ for all $x \in P$. 

\begin{theorem}
    Permissible projective transformations preserve affine slices.
    \label{th:projective-transformations}
\end{theorem}

\begin{proof}
    If $P$ is a polytope and $\varphi$ is a permissible projective transformation for $P$, then $\varphi(P)$ is a polytope combinatorially equivalent to $P$. 
   Let $L=\{x \mid c^\top x + \gamma = 0\}$, and $H \neq L$ be some affine hyperplane. Then the closure $\overline{\varphi(H\setminus L)}$ is an affine hyperplane. If $L^+ = \{x \mid c^\top x + \gamma > 0 \}$ is the open halfspace containing $P$, then $\varphi(H \setminus L) \cap \varphi(L^+)= \overline{\varphi(H \setminus L)} \cap \varphi(L^+)$.
   By injectivity, and with $P \subset L^+$ we get
   \[
   \varphi(P \cap H) = \varphi(P) \cap \varphi(H \setminus L) = \varphi(P) \cap \overline{\varphi(H \setminus L)} \ .
   \]
This shows that $\varphi(P \cap H)$ is a slice of $\varphi(P)$.
Since $\varphi$ is still permissible for the polytope $P \cap H$, it follows that $P \cap H \simeq \varphi(P \cap H)$, which proves the claim.
\end{proof}

\begin{remark}
    Linear, affine and projective transformations are fairly natural slice-preserving maps, which preserve the combinatorics of both $P$ and its slices. A less intuitive map is as follows:
    Let $P$ be a polytope with vertices $v_1,\dots,v_n$, containing the origin in its interior. Let $\varepsilon_i > 0$ and let $Q = \conv(\varepsilon_1 v_1,\dots,\varepsilon_n v_n)$. Then there is a continuous deformation $\varphi$ such that $\varphi(\zero) = \zero$, $\varphi(v_i) = \varepsilon_i v_i$ and $\varphi(P) = Q$ which leaves hyperplanes through the origin invariant. In some cases, such a map is slice-preserving, even if $P$ and $Q$ have distinct combinatorial types. See \Cref{ex:pointwise_cse} for an example of $P$ and $Q$.
\end{remark}

\label{subsec:constructions}

We now turn our attention from transformations to investigating slices of standard constructions of polytopes, namely, prisms, pyramids, and bipyramids. Our goal is to describe the slices of these constructions in terms of the base polytope. 
We begin with prisms and generalized prisms. Recall that the (standard) \emph{prism} over a polytope $P \subset \R^d$ is the polytope $\prism(P) = P \times [-1,1] = \conv(P\times \{1\}, P \times \{-1\})\subset \R^{d+1}$.
As we will show, slices of prisms are given by slabs of $P$, which are intersections of $P$ with a region bounded by two parallel hyperplanes.

\begin{definition}[Slabs]
    A \emph{slab} of a polytope $P \subset \R^d$ is the polytope
    $$\slab_{u, \beta, t}(P) := P \cap \left\{x \in \R^d \ \middle| \  | \scalar{x}{u} + \beta | \leq t \right\}$$
    for some $u \in \R^d$ and $\beta, t \in \R$, where $t \geq 0$.
\end{definition}

For a hyperplane $H = H_u(\beta) = \{ x \in \R^d \mid \langle x,u \rangle + \beta = 0\}$ we define the open halfspace $H^+ := \{x \in \R^d \mid \langle u, x \rangle + \beta > 0 \}$, and the closed halfspace is then $\overline{H^+}$. In this notation, we have
\[
    \slab_{u, \beta, t}(P) = P \cap \overline{H_{u}(\beta+t)^+} \cap  \overline{H_{u}(\beta-t)^-} \ .
\]

\begin{proposition}[Affine slices of prisms]
    Affine slices of $\prism(P)$ are combinatorially equivalent to slabs of $P$ or to prisms over affine slices of $P$.
    \label{prop:prism-slices}
\end{proposition}

\begin{proof}
    We consider $\prism(P) \cap H_u(\beta)$ and denote $u = (u',u_{d+1})$. Without loss of generality $u_{d+1} \geq 0$. If $u_{d+1} = 0$, then
    \[
        \prism(P) \cap H_u(\beta)
        = 
        \{(x,y) \in P \times [-1,1] \mid \langle x,u'\rangle + \beta = 0\} = \prism(P \cap H_{u'}(\beta))\ ,
    \]
    so the slice is a prism over an affine slice of $P$.
    We now show that if $u_{d+1} > 0$, then the projection of $\prism(P) \cap H_u(\beta)$ to the first $d$ coordinates is $\slab_{u',\beta,u_{d+1}}(P)$. Let $\pi: (x_1,\dots,x_d,x_{d+1}) \mapsto (x_1,\dots,x_d)$ denote the projection onto the first $d$ coordinates. Since $u_{d+1} \neq 0$, the restriction $\pi|_{H_u(\beta)} : H_u(\beta) \to \R^d$ is a bijection. 
    Since $\prism(P) = \{(x,\mu) \mid x \in P, \mu \in [-1,1]\}$, we obtain
    \begin{align*}
        (x,\mu) \in \prism(P) \cap H_{u}(\beta) &\iff 
        x \in P, \mu \in [-1,1] \text{ and } \langle x, u' \rangle + \mu u_{d+1} + \beta = 0 \\ &\iff
        x \in P \text{ and } \mu = -\frac{\langle x,u'\rangle + \beta}{u_{d+1}} \in [-1,1] \\ &\iff
        x \in P \text{ and } \langle x,u'\rangle + \beta \in [-u_{d+1},u_{d+1}] \\ &\iff x \in \slab_{u',\beta,u_{d+1}}(P) \ .
    \end{align*}
    It follows that $\pi(\prism(P)\cap H_u(\beta)) = \slab_{u',\beta,u_{d+1}}(P)$. Since $\pi$ is linear and bijective when restricted to $H_u(\beta)$, the slice of $\prism(P)$ is combinatorially equivalent to a slab of $P$.
\end{proof}

\begin{remark}[Affine slices of generalized prisms]\label{prop:generalized-prism-slices}
    For a polytope $P$ containing the origin in its relative interior, we can generalize the notion of prisms to define a \emph{generalized prism} over $P$ to be ${\prism(P, \lambda, \delta) = \conv( \delta P \times \{-1\}, \lambda P\times\{1\})}$ for some $\lambda, \delta >0$.
    This construction will be needed in \Cref{sec:slicing-problems}.
    We note that any such generalized prism is the image of a prism under the permissible projective transformation 
    $$
    (x,t) \mapsto \frac{1}{(\delta - \lambda) t + \delta + \lambda} (2 \delta \lambda x, (\delta + \lambda)t + \delta - \lambda) \ .
    $$
    $$(x_1, \ldots, x_{d+1}) \mapsto (x_1 + \frac{1}{2}(\lambda - \delta)x_{d+1} + \frac{1}{2}(\lambda + \delta)-1, \ldots, x_d + \frac{1}{2}(\lambda - \delta)x_{d+1} + \frac{1}{2}(\lambda + \delta)-1, x_d) \ .$$
    Thus, \Cref{th:projective-transformations} and \Cref{prop:prism-slices} imply that affine slices of generalized prisms are combinatorially equivalent to slabs of $P$ or to generalized prisms over affine slices of $P$.
\end{remark}

However, the following example shows that the description of \Cref{prop:prism-slices} does not hold for all combinatorial prisms.

\begin{example}[Slices of prisms depend on the geometric realization]
    Consider the $3$-dimensional combinatorial bipyramids $P$ and $Q$ from \Cref{ex:perturbed_bipyramid}. Since $P \simeq Q$, we have $\prism(P) \simeq \prism(Q)$. Thus, $\prism(Q)$ is also a combinatorial prism over $P$. We now argue that $\prism(Q)$ has a slice which is neither a slab of $P$ nor a prism over a slice of $P$. Let $H \subset \R^3$ be a hyperplane such that $Q \cap H$ is a polygon with $12$ vertices. Then $H' = H \times \R$ is a hyperplane in $\R^4$, and $\prism(Q) \cap H' = \prism(Q \cap H)$ is a prism over a $12$-gon. It has $14$ facets and two of its facets are $12$-gons. 
    To show that $\prism(Q) \cap H'$ is not combinatorially equivalent to a prism over a slice of $P$, note that slices of $P$ have at most $8$ vertices, so a prism over a slice of $P$ has at most $10$ facets. To show that $\prism(Q) \cap H'$ is not combinatorially equivalent to a slab of $P$, we claim that slabs of $P$ have facets with at most $8$ vertices, hence can't be combinatorially equivalent to a prism over a $12$-gon. Indeed, facets of slabs of $P$ are either slabs of facets of $P$, or slices of $P$. Facets of $P$ are triangles, so slabs of triangles have at most $5$ vertices. Slices of $P$ have at most $8$ vertices. Thus  \Cref{prop:prism-slices} does not hold for combinatorial prisms in general.
\end{example}

We next discuss another classic polytope construction, namely pyramids. Recall that the \emph{(standard) pyramid} over a polytope $P \subset \R^d$ is the polytope $\pyr(P) = \conv( P \times \{0\}, e_{d+1}) \subset \R^{d+1}$, and any polytope which is combinatorially equivalent to $\pyr(P)$ is a \emph{pyramid} over $P$. Since a standard pyramid differs from a prism only by collapsing one of the two copies of the base to a single point, the description of its slices is closely related to the prism case. Moreover, any combinatorial pyramid with base facet $F$ is affinely equivalent to $\pyr(F)$. Thus, by \Cref{lem:affine-transformations}, the following characterization of slices of pyramids applies to all combinatorial pyramids.

\begin{proposition}[Affine slices of pyramids]\label{prop:pyramid-slices}
    Let $\pyr(P)$ be the pyramid over a polytope $P \subset \R^d$. Affine slices of $\pyr(P)$ are combinatorially equivalent to a slab of $P$ or to the pyramid over an affine slice of $P$.
\end{proposition}

\begin{proof}
    Let $H$ be a hyperplane in $\R^{d+1}$. 
    If $H$ does not intersect the relative interior of $\pyr(P)$, then $\pyr(P) \cap H$ is a face of $\pyr(P)$. If this face is contained in the base $P \times \{0\}$, then it is a face $F$ of $P$, and hence a slab of $P$ obtained by intersecting $P$ with a supporting hyperplane. If the face contains the apex $e_{d+1}$ and is not the apex itself, then it is of the form $\pyr(F)$ for some face $F$ of $P$. Since $F$ is an affine slice of $P$, this is a pyramid over an affine slice of $P$. Finally, if $\pyr(P) \cap H = \{e_{d+1}\}$, then the slice is combinatorially equivalent to a vertex of $P$, and hence to a slab of $P$. 
    
    Next, suppose  that $H$ intersects the relative interior of $\pyr(P)$. If $H$ contains the apex $e_{d+1}$ of the pyramid, then $\pyr(P) \cap H = \pyr(P \cap H)$ is a pyramid over a slice of $P$.
    Otherwise, there exists $\lambda$ such that $\pyr(P) \cap H$ is combinatorially equivalent to a slice of $\prism(P,\lambda)$, that is, we can cut off the apex and obtain the slice as a slice of a generalized prism (up to dilation and translation). From the proof of \Cref{prop:prism-slices} and \Cref{prop:generalized-prism-slices} it follows that $\pyr(P) \cap H$ is combinatorially equivalent to a slab of $P$.
   \end{proof}

In contrast to pyramids, for bipyramids the geometric realization again influences the slicing-behavior. Indeed, as illustrated by
\Cref{ex:perturbed_bipyramid}, combinatorially equivalent bipyramids can admit different combinatorial types of slices.
Moreover, the combinatorics of slices may change if the base polytope is fixed and only one of the apices is perturbed.
Thus, as for prisms, statements about a standard bipyramid do not automatically extend to all combinatorial realizations. Rather than giving a complete description of the slices of standard bipyramids, we focus here on a lower bound for the face numbers of their central slices in terms of the base polytope. This will be used in the construction of the counterexample to the combinatorial Bourgain's slicing problem in \Cref{sec:slicing-problems}.

We denote by $f(P) = (f_{-1}(P), f_0(P),\dots,f_{d-1}(P))$ the $f$-vector of $P$ with $d = \dim(P)$, whose entries  $f_k(P)$ record the number of $k$-dimensional faces of $P$, and $f_{-1}(P)=1$ counts the empty face.
Note that the $f$-vector of the bipyramid satisfies
\[
f_k(\bipyr(P)) = \begin{cases}
    f_k(P) + 2f_{k-1}(P) & \text{for } 0 \leq k \leq d-1 \\
    2f_{k-1}(P) & \text{for } k=d \ .
\end{cases}
\]

\begin{proposition}[Central slices of bipyramids]
\label{prop:slices_bip}
Let $P \subset \R^d$ be a $d$-dimensional polytope containing the origin in its interior. For $0 \leq k \leq d-2$, let $s_k \in \N$ be a constant such that for all $v \in \Sph^{d-1}$ it holds that $f_k(P \cap v^\perp) \geq s_k$. Formally, we define $s_{-1} = 1$ and $s_{d-1}=0$. Then for all $0 \leq k \leq d-1$ and all $u \in \Sph^{d}$ we get $f_k(\bipyr(P) \cap u^\perp) \geq \min(f_k(P), s_k +2s_{k-1})$.
\end{proposition}

\begin{proof}
    We know that $f_k(\bipyr(Q)) = f_k(Q)+2f_{k-1}(Q)$ for any $d$-polytope $Q$ and $0 \leq k \leq d-1$, and $f_{d}(\bipyr(Q)) = 2f_{d-1}(Q)$.
    Let $u \in \Sph^{d}$. If $e_{d+1} \in u^\perp$, then $\bipyr(P) \cap u^\perp=\bipyr(P \times \{0\} \cap u^\perp)$, i.e., this slice is a bipyramid over a slice of $P$. Since the origin lies in the relative interior of $P$, the slice $(P \times \{0\}) \cap u^\perp$ has dimension $d-1$. Thus, we have
    \[ f_k(\bipyr(P) \cap u^\perp) =  f_k(P \times \{0\} \cap u^\perp) + 2 f_{k-1}(P \times \{0\} \cap u^\perp) \geq s_k + 2s_{k-1}\]
    for $0 \leq k\leq d-2$, by our premises. For $k=d-1$, we have
    \[
    f_{d-1}(\bipyr(P) \cap u^\perp)
    =
    2 f_{d-2}(P \times \{0\} \cap u^\perp)
    \geq
    2s_{d-2}
    =
    s_{d-1}+2s_{d-2} \ .
    \]
    If $e_{d+1} \not \in u^\perp$, then $e_{d+1}, -e_{d+1}$ lie in different open halfspaces of $u^\perp$. Let $F$ be a $k$-face of $P \times \{0\}$. We claim that $u^\perp$ intersects at least one of the $(k+1)$-faces $\conv(F, e_{d+1})$ or $\conv(F, -e_{d+1})$ of $\bipyr(P)$ in codimension 1. This is easy to see: If $u^\perp$ contains $F$ or $u^\perp$ intersects the relative interior of $F$, then it also intersects $\conv(F,e_{d+1})$ and $\conv(F,-e_{d+1})$ in codimension $1$, and the claim holds. In the only remaining case, some or all vertices of $F$ are contained in an open halfspace of $u^\perp$ (and some may lie on $u^\perp$). Then, either $e_{d+1}$ or $-e_{d+1}$ lies in the other open halfspace, thus, either $\conv(F, e_{d+1})$ or $\conv(F, -e_{d+1})$ is intersected by $u^\perp$ in its relative interior. 
    This gives an injection from the $k$-faces of $P$ to the $k$-faces of $\bipyr(P) \cap u^\perp$,  so $f_k(\bipyr(P) \cap u^\perp) \geq f_k(P)$.
    Combining both cases, we obtain
    \[ f_k(\bipyr(P) \cap u^\perp)  \geq \min(f_k(P), s_k+2s_{k-1}) \ . \qedhere \]
\end{proof}

\section{Reconstruction of combinatorial types from central slices}\label{sec:reconstruction-central}

In \Cref{sec:lattices,subsec:constructions} we observed that the combinatorics of the polytope determine, to a certain extent, the combinatorics of its slices. We now turn to the inverse question: to what extent can the combinatorics of a polytope be reconstructed from combinatorial data of its slices?
First, we show that simplices are completely characterized by merely considering the list of combinatorial types of their slices. Then,
we show that the combinatorics of certain classes of polytopes are uniquely determined already by the coarsest combinatorial information of their central slices, namely, the number of vertices of their slices. 
More precisely, we consider the function recording the number of vertices of a slice:
\begin{align*}
    \varphi_P : \text{affine / central hyperplanes in } \R^d &\to \N \\
    H &\mapsto f_0(P \cap H) \ .
\end{align*}
When the domain of this function is the space of all affine hyperplanes in $\R^d$, then, in particular, $\varphi_P^{-1}(0)$ is the set of all hyperplanes
contained in $\R^d \setminus P$, which allows us to uniquely recover $P$.
Thus, not only the combinatorics but the precise geometric realization of $P$ is uniquely determined by this function.

We begin by characterizing simplices from the combinatorial types of their slices.
Let $\Delta_p$ denote a $p$-dimensional simplex.
Recall that the \emph{vertex figure} $P/v$ of a polytope $P$ at a vertex $v$ is the combinatorial type of the intersection of $P$ with a generic hyperplane separating $v$ from all other vertices of $P$.
\begin{theorem}
    Let $d\geq 3$. A $d$-dimensional polytope $P \subset \R^d$ is a simplex if and only if every nonempty slice is combinatorially equivalent to an $r$-fold pyramid over a product of simplices $
		\pyr^{r}(\Delta_p \times \Delta_q),
		$
    for some $p,q,r\geq 0$ satisfying $p+q+r=d-1$, or to a simplex of dimension at most $d-2$.
\end{theorem}

\begin{proof}
    Suppose $P$ is a simplex and $H$ an affine hyperplane. If $\dim(P \cap H)\leq d-2$, then $P\cap H$ is a face of $P$, and hence a simplex of dimension at most $d-2$. Suppose that $\dim(P \cap H)=d-1$.  Since $P$ has $d+1$ facets, the slice $P \cap H$ is a $(d-1)$-dimensional polytope with at most $d+1$ facets. By \cite[Sec.~6.1]{Grunbaum03:ConvexPolytopes}, any $(d-1)$-polytope with at most $d+1$ facets is combinatorially equivalent to 
    $\pyr^{r} \bigl(\Delta_p\times \Delta_q\bigr),$
    for some $p,q,r\geq 0$ satisfying $p+q+r=d-1$. $P\cap H$ is thus combinatorially equivalent to $\ \pyr^r \bigl(\Delta_p\times \Delta_q\bigr).$ 

    Conversely, suppose that every slice is of this form.
    To show that $P$ is simplicial, let $F$ be a facet of $P$, and let $F_1,\dots,F_m$ be the neighboring facets, i.e., $\dim(F \cap F_i) = d-2$. Let $H'$ be a hyperplane parallel to $F$, which intersects the interior of $P$ close to $F$. To obtain a hyperplane which intersects the relative interiors of $F,F_1,\dots,F_m$, we tilt $H'$ (sufficiently little)
    to cross a single vertex in $F \cap F_1$. 
    Denote this hyperplane by $H$.
    Now, $H$ intersects the relative interiors of $F$ and $F_1,\dots,F_m$.
    Thus, $P \cap H$ is a $(d-1)$-dimensional polytope with $m+1$ facets $F\cap H, F_1\cap H,\dots, F_m \cap H$. By assumption, $P \cap H \simeq \pyr^{r}(\Delta_p \times \Delta_q)$ for some $p,q,r \geq 0$ such that $p+q+r = d-1$. Hence, the number of facets of $P \cap H$ is 
    \[
        m+1 = 
        \begin{Bmatrix} r + ((p+1) + (q+1)) & \text{if } p,q > 0 \\
        r + p + q + 1 & \text{ otherwise}
        \end{Bmatrix} \leq  d+1 \ , 
    \]
     so $F$ has $m\leq d$ neighboring facets. As $F$ is a $(d-1)$-dimensional polytope with at most $d$ facets, $F$ is a simplex. This implies that $P$ is simplicial.

     To show that $P$ is simple, consider a vertex figure $P/v$ at a vertex $v \in \verts(P)$. Since $P$ is simplicial, so is the vertex figure. To see this, note that the facets of a vertex figure $P/v$ are vertex figures $F/v$ of facets $F$ of $P$ incident to $v$, and that the vertex figure of a simplex $F$ is a simplex. We thus have that every proper face of the vertex figure $P/v$ is a simplex. Assume for contradiction that $P/v$ is not a simplex itself. Since every vertex figure arises as a slice, we have that the vertex figure is of the form $\pyr^r (\Delta_p \times \Delta_q)$. If $r>0$, then the vertex figure has a proper $(p+q)$-dimensional face $\Delta_p \times \Delta_q$. Since $P/v$ is simplicial, we have that $\Delta_p \times \Delta_q$ is a simplex, and hence $p=0$ or $q=0$. Thus, we can assume that $q=0$ and we have that the vertex figure is the simplex $\pyr^r(\Delta_p \times \Delta_0) = \Delta_{p+r}$; a contradiction. 
     Next, assume that $r = 0$. Since $P/v$ is assumed not to be a simplex, we have $p,q \geq 1$. However, if $p > 1$ or $q>1$ then $P /v \simeq \Delta_p \times \Delta_q$ is not simplicial. Thus, $p=q=1$, which implies in particular that $d=3$ and that $P/v$ is a $2$-dimensional quadrilateral.      
     Hence, $\deg(v) = 4$. Let $w \in \verts(P)$ be a vertex incident to $v$. Consider the generic hyperplane $H$ separating $v$ and $w$ from the remaining vertices $\verts(P) \setminus \{v,w\}$. Since $\deg(w) \geq 3$, we have that $P \cap H$ intersects $3$ edges incident to $v$ and at least $2$ edges incident to $w$. Thus, $P \cap H$ has at least $5$ vertices. However, note that in dimension $d=3$, the possible slices are of the form $P \cap H \simeq \pyr^r(\Delta_p \times \Delta_q) \in \{\Delta_2, \Delta_1 \times \Delta_1\}$, so $f_0(P \cap H) \leq 4$, yielding a contradiction. It follows that every vertex figure is a simplex, so $P$ is simple. Thus, $P$ is simple and simplicial, and therefore a simplex.
\end{proof}

We now turn to the reconstruction of simple polytopes from the number of vertices of their slices.
For the remainder of this section, we therefore focus on the restriction to central hyperplanes, and assume that the origin is contained in the interior of any polytopes considered. 
A key object associated to central slices is the central hyperplane arrangement
\[
    \mathcal H(P) = \{v^\perp \mid v \text{ is a vertex of $P$ and not the origin} \}  
\]
\cite{Berlow2022,Brandenburg2025}.
To see the connection, recall from \Cref{sec:lattices} that the combinatorial type of the central slice $P \cap u^\perp$ is uniquely determined by the vertices it contains and by the edges it intersects, since both $\mathcal F^\circ$ and $\mathcal F^\subset$ are determined by their minimal elements (see \Cref{rem:minimal-elements}). The hyperplane $u^\perp$ contains the vertex $v \in \verts(P)$ if and only if $u \in v^\perp$, and the edge with vertices $v,w$ is intersected if and only if $u \in (v^\perp)^+ \cap (w^\perp)^-$ or $u \in (v^\perp)^- \cap (w^\perp)^+$. Thus, the combinatorial type of the slice $P\cap u^\perp$ is determined by the relatively open chamber of $\mathcal H(P)$ which contains $u$. Next, we show that for any polytope $P$ such that $0 \in \inter(P)$, this governing hyperplane arrangement is determined by the function $\varphi_P(u) = f_0(P \cap u^\perp)$.

\begin{proposition}\label{prop:reconstruct-arrangement}
    Let $P$ be a $d$-dimensional polytope containing the origin in its interior.
    If $d \geq 3$, then the hyperplane arrangement $\mathcal H(P)$ is uniquely determined by the function $\varphi_P: \Sph^{d-1} \to \N$ given by $ \varphi_P(u) = f_0(P \cap u^\perp)$. 
\end{proposition}

\begin{proof}
    We show that we can reconstruct the hyperplane arrangement $\mathcal H(P)$ from the locus of non-differentiability of $\varphi_P$. Let $H$ be a hyperplane of $\Hcal(P)$ and $u \in H$ be a vector lying in no other hyperplane of $\mathcal H(P)$.
    Then either we have that for all such $u$, the corresponding hyperplane $u^\perp$ intersects $P$ in precisely $k=1$ vertex, or all such hyperplanes intersect in precisely $k=2$ (antipodal) vertices. 
    Let $E$ be the set of edges which are intersected by $u^\perp$ in their relative interiors. Then $\varphi_P(u) = |E| + k$. Any 
    open full-dimensional cell of $\mathcal H(P)$ that is neighboring $H$ corresponds to a family of central hyperplanes $(u')^\perp$ which intersect the edges $E$ and additionally some edges $E(u')$ incident to the vertex or vertices contained in $u^\perp$. If $k=1$, since $d \geq 3$ there is at least one open full-dimensional cell on which $|E(u')| > 1$, so $\varphi_P(u) = |E| + |E(u')| > |E| + k$. If $k=2$, then there is at least one open full-dimensional cell on which $|E(u')|>2$, so $\varphi_P(u') = |E| + |E(u')| > |E| + k$.
    This implies that piecewise-constant function $\varphi_P$  ``changes value'', i.e., is non-differentiable at the hyperplanes of the arrangement $\mathcal H(P)$ restricted to the sphere $\Sph^{d-1}$. Since $\varphi_P$ is constant along interiors of cells of $\Hcal(P)$, it follows that the locus of non-differentiability equals the hyperplanes contained in $\mathcal H(P) \cap \Sph^{d-1}$. We can thus reconstruct the arrangement $\mathcal H(P)$ from the function $\varphi_P$.
\end{proof}

\begin{definition}
    A set $V \subset \R^d$ is in \emph{3-general linear position} if every $3$-element subset of $V$ is linearly independent. A polytope is in $3$-general linear position if its vertices are in $3$-general linear position.
\end{definition}

 This notion of 3-general linear position forbids \emph{linear} dependence of $3$ points, meaning that no $3$ points lie in a common plane through the origin. In particular, this property is not translation invariant. If a $3$-dimensional polytope $P$ is not in $3$-general linear position then almost all translations of $P$ will put it in 3-general linear position.

\begin{theorem}\label{prop:graph-reconstruction}
    Let $P$ be a simple $d$-dimensional polytope in $3$-general linear position containing the origin in its interior.
    If $d \geq 3$, then the graph of $P$ is uniquely determined by the function $\varphi_P: \Sph^{d-1} \to \N$ given by $ \varphi_P(u) = f_0(P \cap u^\perp)$. 
\end{theorem}

\begin{proof}
    By \Cref{prop:reconstruct-arrangement}, we can reconstruct the arrangement $\mathcal H(P)$ from the function $\varphi_P$. Since $P$ is in 3-general linear position, each hyperplane of $\mathcal H(P)$ corresponds to a unique vertex of $P$. Thus, we have a bijection between hyperplanes of $\mathcal H(P)$ and vertices $V(G)$ of the graph $G$ of $P$.
    Let $H_v,H_w$ be two distinct hyperplanes with corresponding vertices $v,w \in V(G)$.
    We want to recover if there is an edge between $v$ and $w$ or not. For this, consider a $(d-2)$-dimensional open cell $C_{vw}$ of $\mathcal H(P)$ which is supported on $H_v \cap H_w$, and let $u_{vw} \in \Sph^{d-1}$ be contained in this cell. Then $u_{vw}^\perp$ is a hyperplane containing the vertices $v,w$, but the assumption of $3$-general linear position implies that no further vertex of $P$ is contained in $u_{vw}^\perp$. In particular, if $E$ now denotes the edges which are properly intersected by $u_{vw}^\perp$ (and not contained in $u_{vw}^\perp$) then $\varphi_P(u_{vw}) = |E| + 2$. This allows us to compute $|E|$. 
    We will detect whether there is an edge between $v$ and $w$ or not by ``wiggling'' $u_{vw}$ inside $w^\perp$, as follows. Let $C_{vw}^+, C_{vw}^-$
    denote the unique $(d-1)$-dimensional cells which are contained in $w^\perp$ and share a common facet $C_{vw}$. 
    Even though we do not know the precise coordinates of $v$, we can choose the labeling of these cells such that
     \[
       C_{vw}^+ \subset \{ x \in w^\perp \mid \langle x,v \rangle > 0\}\ , \qquad C_{vw}^- \subset \{ x \in w^\perp  \mid \langle x,v \rangle < 0\}\  .
    \]
    Pick $u_{vw}^+ \in C_{vw}^+ \cap\Sph^{d-1}, u_{vw}^- \in C_{vw}^- \cap \Sph^{d-1}$. By construction, $(u_{vw}^+)^\perp$ and $(u_{vw}^-)^\perp$ both contain the point $w$ and do not contain the point $v$. 
    Let $E_{vw}^+$ denote the edges which are properly intersected by $(u_{vw}^+)^\perp$ and are incident to $v$. Similarly, let $E_{vw}^-$ denote the edges which are properly intersected by $(u_{vw}^-)^\perp$ and are incident to $v$. 
    By construction, we have
    \[
        E_{vw}^+ = \{ vn \text{ edge } \mid \langle u^+_{vw}, n \rangle < 0 \}, \qquad E_{vw}^- = \{ vn \text{ edge } \mid \langle u^-_{vw}, n \rangle > 0 \} \ .
    \]
    Note that these sets are disjoint.
    Since $w$ is the only vertex contained in the hyperplanes $(u_{vw}^+)^\perp$ and $(u_{vw}^-)^\perp$, we have 
    \[
        \varphi_P(u_{vw}^+) = |E| + |E_{vw}^+| + 1, \qquad \varphi_P(u_{vw}^-) = |E| + |E_{vw}^-| + 1 \ .
    \]
     We can therefore recover $|E_{vw}^+|$ and $|E_{vw}^-|$ uniquely from $\varphi_{P}(u_{vw}), \varphi_P(u_{vw}^+),$ and $\varphi_P(u_{vw}^-)$. 

    We are left with checking a few cases for the sizes of these sets.
    If $|E_{vw}^+|+|E^-_{vw}| = d$, then they form a partition of all edges incident to the simple vertex $v$. But since 
    $(u_{vw}^+)^\perp$ and $(u_{vw}^-)^\perp$
    contain $w$ but not $v$, this implies that there is no edge between $v$ and $w$.
    On the other hand, if $|E_{vw}^+|+|E^-_{vw}| < d$ then we show that $|E_{vw}^+|+|E^-_{vw}| = d-1$, and that there is an edge between $v$ and $w$. Indeed, if there is no edge then the construction of $E_{vw}^+$ and $E_{vw}^-$ implies that there is some neighbor $n$ of $v$ such that $\langle u_{vw}, n \rangle = 0$, i.e., that $n \in u_{vw}^\perp$. However, this violates that $P$ is in 3-general linear position, i.e., that the only vertices contained in $u_{vw}^\perp$ are $v$ and $w$, which yields a contradiction. Thus, $vw$ forms an edge, and this is the only edge incident to $v$ which is not contained in $E_{vw}^+ \sqcup E_{vw}^-$. We thus have
    \[
        |E^+_{vw}| + |E^-_{vw}| = 
        \varphi_P(u_{vw}^+) + \varphi_P(u_{vw}^-) - 2\varphi_P(u_{vw}) +2 =  
        \begin{cases}
            d-1 & \text{if } vw \text{ is an edge} \\
            d & \text{otherwise} \ ,
        \end{cases}
    \]
    and we
    can uniquely recover from $|E_{vw}^+|$ and $|E_{vw}^-|$ if there is an edge between $v$ and $w$.
\end{proof}

Note that the drums from \Cref{ex:drums} are not simple, nor in 3-general linear position. They have the same vertex-counting function $\varphi:\Sph^{d-1} \to \N$ as in \Cref{prop:graph-reconstruction}, but distinct graphs. Reconstructing the graph of a simple polytope allows us to reconstruct the entire combinatorial type, as follows.

\begin{theorem}\label{th:simple-polytope-reconstructible}
   Let $P$ be a simple $d$-dimensional polytope  in 3-general linear position and containing the origin in its interior. Then the combinatorial type of $P$ is reconstructible from the function $\varphi_P: \Sph^{d-1} \to \N$ given by $ \varphi_P(u) = f_0(P \cap u^\perp)$. 
\end{theorem}

\begin{proof}
    By \Cref{prop:graph-reconstruction}, the graph of $P$ is reconstructible from $\varphi_P$. Since any simple polytope is reconstructible from its graph \cite{Blind1987,Kalai1988}, the claim follows.
\end{proof}

We now move to centrally symmetric polytopes. Note that if $V \subset \R^d$ is centrally symmetric (about the origin), i.e., $V=-V$, then it is not in general linear position. However, by introducing a symmetric notion of 3-general linear position, we get an analogous reconstruction for centrally symmetric polytopes.

\begin{definition}
    A centrally symmetric set $V \subset \R^d$ is in \emph{symmetric 3-general linear position} if for every centrally symmetric $6$-element subset $S \subseteq V$ it holds that $\dim(\operatorname{span}(S))=3$. A centrally symmetric polytope is in \textit{symmetric 3-general linear position} if its vertices are in symmetric 3-general linear position.
\end{definition}

Equivalently, writing a centrally symmetric set $V$ not containing the origin as a disjoint union $V=V_0 \sqcup -V_0$, $V$ is in symmetric 3-general linear position if and only if $V_0$ is in 3-general linear position.

\begin{theorem}\label{cor:cs-simple-polytope-reconstructible}
    Let $P$ be a centrally symmetric simple $d$-dimensional polytope in symmetric $3$-general linear position. If $d \geq 3$, then the combinatorial type of $P$ is reconstructible from the function $\varphi_P: \Sph^{d-1} \to \N$ given by $ \varphi_P(u) = f_0(P \cap u^\perp)$. 
\end{theorem}

\begin{proof}
    By \Cref{prop:reconstruct-arrangement}, we can reconstruct the arrangement $\mathcal H(P)$ from the function $\varphi_P$. Since $P$ is in symmetric 3-general linear position, each hyperplane of $\mathcal H(P)$ corresponds to a unique pair of antipodal vertices of $P$. Thus, we have a bijection between hyperplanes of $\mathcal H(P)$ and antipodal pairs of vertices $V(G)$ of the graph $G$ of $P$.
    Let $H_v,H_w$ be two distinct hyperplanes with corresponding vertices $\pm v, \pm w \in V(G)$.
    We assign an orientation to $H_v$ such that $H_v^+ = \{u \mid \langle u,v \rangle > 0\}$, and similarly for $H_w^+$.

    First, we show that $vw$ and $v(-w)$ cannot both be edges simultaneously. Assume for contradiction that $vw$ and $v(-w)$ are both edges. Since $P$ is simple, the two edges are contained in a $2$-dimensional face $F$ of $P$. Since $F$ has vertices $w,-w$, $F$ contains the origin. By assumption, $P$ contains the origin in its relative interior, so $P = F$, a contradiction to $\dim(P) \geq 3$.

    Thus, it remains to determine whether there are edges $vw$ and $(-v)(-w)$, or $v(-w)$ and $(-v)w$, or no edge between any of these vertices. 
    For this, consider a $(d-2)$-dimensional open cell $C_{vw}$ of $\mathcal H(P)$ which is supported on $H_v \cap H_w$.
    Let $u^{++}, u^{+-}, u^{-+}, u^{--} \in \Sph^{d-1}$ be contained in full-dimensional adjacent cells of $C_{vw}$, such that \[
    u^{++} \in H_v^+ \cap H_w^+, \ u^{+-} \in H_v^+ \cap H_w^-, \ u^{-+} \in H_v^- \cap H_w^+, \ u^{--} \in H_v^- \cap H_w^- \ .
    \]
    Note that for each of these choices $u \in \{u^{++}, u^{+-}, u^{-+}, u^{--}\}$, the hyperplane $u^\perp$ does not contain any vertices. Hence, $\varphi_P(u)$ is the number of edges which are intersected by $u^\perp$. We compute the value of 
    \begin{align}\label{eqn:varphi_ppmm}
         \varphi_P(u^{++}) + \varphi_P(u^{--}) - \varphi_P(u^{+-}) - \varphi_P(u^{-+}) \ .
    \end{align}
    First, suppose that $vn$ is an edge and $n \not \in \{w,-w\}$. Since all $u \in \{u^{++}, u^{+-}, u^{-+}, u^{--}\}$ lie in maximal cells adjacent to $C_{vw}$, and $P$ is in symmetric $3$-general linear position, we have that they all lie in the same open halfspace defined by $H_n$. 
    Indeed, since $P$ is in symmetric $3$-general linear position, the points $v,w,n$ are linearly independent, and hence $H_v \cap H_w \not \subseteq H_n$. Thus, $C_{vw}$ is contained in one of the open half-spaces defined by $H_n$, and $H_n$ is not a bounding hyperplane of the closure of $C_{vw}$. Since $C_{vw}$ is a cell of the arrangement, it follows that the neighboring maximal cells of $C_{vw} \subset H_v \cap H_w$ are contained in the same open half-space of $H_n$.
    
    Thus, 
    $\sgn(\langle u, n\rangle) = c $ for all $u \in \{u^{++}, u^{+-}, u^{-+}, u^{--}\}$ and some fixed $c \in \{-1,1\}$. 
    Since $u^\perp$ intersects the edge $vn$ if and only if $\sgn(\langle u,v\rangle) = - \sgn(\langle u,n\rangle)$, the edge $vn$ contributes exactly to one of the numbers $\varphi_P(u^{++}),\varphi_P(u^{--})$, and exactly to one of the numbers $\varphi_P(u^{+-}),\varphi_P(u^{-+})$. A analogous argument applies to edges incident to $-v, w$ or $-w$. Hence, whether there exist any of the edges $vn$, $wn$, $(-v)n$, $(-w)n$ or not, does not influence the value of \eqref{eqn:varphi_ppmm}. Similarly, if $nm$ is an edge with $n,m\notin\{\pm v,\pm w\}$, then the signs of both $\langle u,n\rangle$ and $\langle u,m\rangle$ are constant for all $u\in\{u^{++},u^{+-},u^{-+},u^{--}\}$. Hence, $nm$ is either intersected by all four hyperplanes or by none of them, and therefore its contribution to \eqref{eqn:varphi_ppmm} is zero.
    
    Now,
    if $vw$ is an edge, then so is $(-v)(-w)$, and both edges are intersected by $(u^{+-})^\perp$ and $(u^{-+})^\perp$, but not by $(u^{++})^\perp$ or $(u^{--})^\perp$. On the other hand, if $v(-w)$ is an edge, then so is $(-v)w$, and both edges are intersected by $(u^{++})^\perp$ and $(u^{--})^\perp$, but not by $(u^{+-})^\perp$ or $(u^{-+})^\perp$. We thus have
    \[ 
    \varphi_P(u^{++}) + \varphi_P(u^{--}) - \varphi_P(u^{+-}) - \varphi_P(u^{-+}) = 
    \begin{cases}
        -4 &\text{if } vw, (-v)(-w) \text{ are edges} \\
        4 &\text{if } (-v)w, v(-w) \text{ are edges} \\
        0 &\text{otherwise} \ ,
    \end{cases}
    \]
    
    so we can uniquely recover the graph of $P$ from $\varphi_P$. Since any simple polytope is reconstructible from its graph \cite{Blind1987,Kalai1988}, the claim follows.
\end{proof}

\section{Hierarchy of slice equivalences}\label{sec:hierarchies}

We now compare several notions of equivalence between polytopes that are defined in terms of their slices. More specifically, we ask how similar or different two polytopes $P, Q$ need to be in order to have similar or different behavior with respect to their slices. There are several natural ways to make this question precise. One may compare slices \emph{pointwise}, requiring that corresponding hyperplanes define combinatorially equivalent slices, or only \emph{setwise}, requiring that the same combinatorial types of slices occur. Moreover, one may consider either all affine hyperplanes or only central hyperplanes.

The resulting notions of equivalence interact with other combinatorial and geometric data of the polytope, such as its face lattice, its oriented matroid, and the hyperplane arrangements determined by its vertices. The purpose of this section is to organize these relations into two hierarchies: one for central slices and one for affine slices. Before doing so, we collect several examples that will be used throughout the section to show that many of the possible implications fail.

\subsection{Pathological examples}\label{sec:pathological-examples}

We collect examples of pairs of polytopes that differ only by minor geometric modifications, resulting in combinatorially different polytopes, while leaving the behavior of the slices unchanged. 
Recall from \Cref{sec:reconstruction-central} that an important object governing central slices is the hyperplane arrangement $\mathcal H(P) = \{v^\perp \mid v \in \verts(P) \text{ and $v$ is not the origin}\}$. Among other properties, we present examples in relation to this arrangement.
All examples presented in this subsection will serve as examples and counterexamples in the hierarchies presented in \Cref{sec:hierarchy-affine,sec:hierarchy-central}. The previously presented \Cref{ex:perturbed_bipyramid} should also be considered as part of the list of examples in this section.

\begin{example}[Same sets of affine and central slices]\label{ex:0sym-setwise-cse-not-comb-equiv}
 This example demonstrates that two polytopes can admit the same set of combinatorial types of slices while being combinatorially distinct, also when we enforce central symmetry of both the polytopes and the slices. Let $P \subset \R^3$ be the bipyramid over a square with a pair of stacked faces and $Q \subset 
 \R^3$ be obtained from the cube by slight perturbation of its vertices, see \Cref{fig:0symm-setwise-cse-not-comb-equiv}.
 One can check, for example computationally, that for both polytopes the set of combinatorial types of affine slices is all $n$-gons for $n\leq 10$, and the set of combinatorial types of central slices is $4,6$ and $8$-gons. It is not clear how, or if, this example generalizes to higher dimensions.
\end{example}
    \begin{figure}[H]
    \centering
    \hspace{8em}%
    \begin{minipage}{0.4\textwidth}
            \includegraphics[width=.7\textwidth]{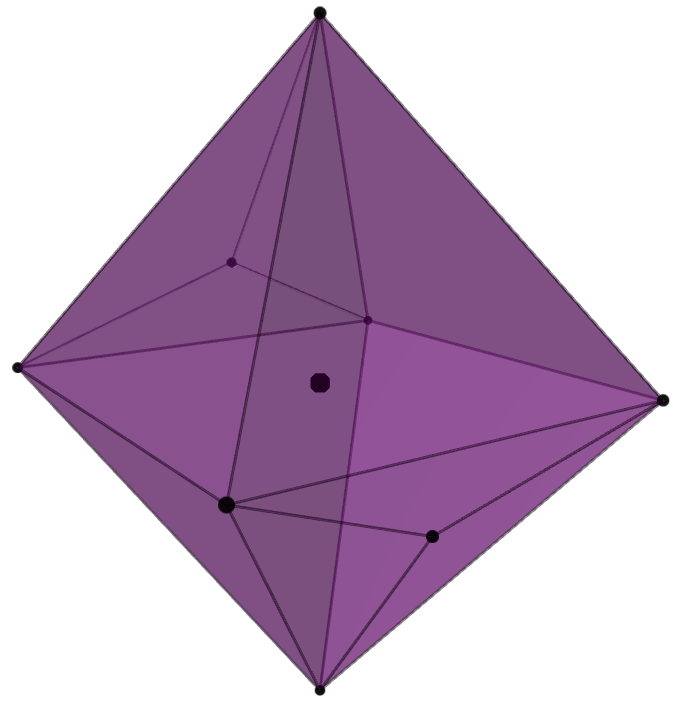}
        \end{minipage}
        \hfill%
        \begin{minipage}{0.4\textwidth}
        \includegraphics[width=.6\textwidth]{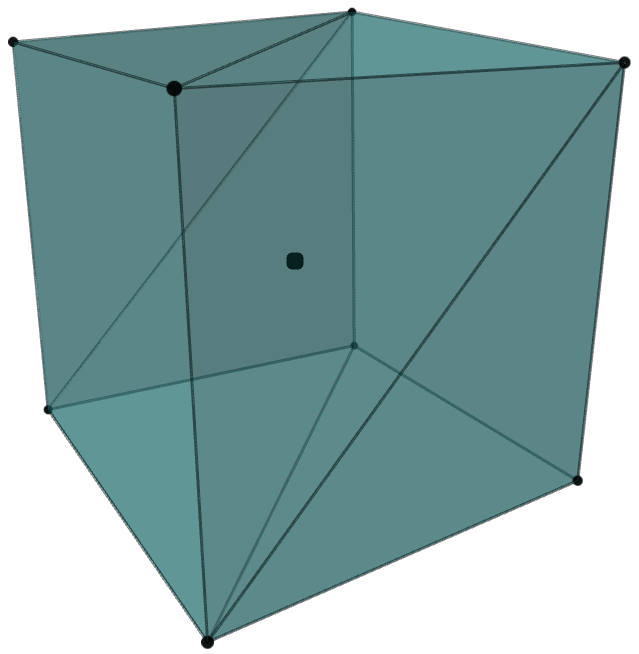}
        \end{minipage}
        \hspace{8em}
    \caption{(Centrally symmetric) polytopes that are combinatorially distinct, but whose combinatorial types of central and affine slices agree (\Cref{ex:0sym-setwise-cse-not-comb-equiv}).}
\label{fig:0symm-setwise-cse-not-comb-equiv}
\end{figure}

\begin{example}[Drums]\label{ex:drums}
    Let $N_{12} = \conv(v_1,\dots,v_{12}) \subset \R^2$ denote a regular $12$-gon centered at the origin, and let $R = \conv(N_{12} \times \{1\}, N_{12} \times \{-1\})$. 
    We modify $N_{12}$ as
    \begin{align*}
        T &= \conv(v_1, v_2, (1+\varepsilon) v_3, v_4, v_5, (1+\varepsilon) v_6, v_7, v_8, (1+\varepsilon) v_9, v_{10}, v_{11}, (1+\varepsilon) v_{12}) \ , \\ 
         T' &= \conv(v_1, v_2, v_3, v_4, v_5, (1+\varepsilon) v_6, v_7, v_8, v_9, v_{10}, v_{11}, (1+\varepsilon) v_{12}) \ , \\
         T'' &= \conv(v_1, v_2, (1+\varepsilon) v_3, v_4, v_5,  v_6, v_7, v_8, (1+\varepsilon) v_9, v_{10}, v_{11},  v_{12}) \ ,
     \end{align*}
     for sufficiently small $\varepsilon > 0$.
    The drums $P$ and $Q$ (as depicted in \Cref{fig:drums}) are obtained as
    \[
        P = \conv \left( T \times \{1\}, N_{12} \times \{-1\} \right) \ , \qquad  Q = \conv \left( T' \times \{1\}, T'' \times \{-1\} \right) \ .
    \]
   Then $P$ and $Q$ are not combinatorially equivalent, however their hyperplane arrangements $\Hcal(P) = \Hcal(Q)$ agree, and pointwise  their central slices agree, that is, $P \cap u^\perp \simeq Q \cap u^\perp$ for every $u \in \Sph^2$.
To see this, note that the combinatorial type of a $2$-dimensional slice is determined by its number of vertices, that is, the number of vertices of $P$ contained in $u^\perp$, and the number of edges of $P$ that are properly intersected by $u^\perp$. $P$ and $Q$ only differ by the direction of diagonals that break certain quadrilateral side faces, ensuring that any central hyperplane $u^\perp$ intersects the same number of vertices and edges in their relative interiors. This construction is special to the non-centrally symmetric setting: in a centrally symmetric polytope, opposite quadrilaterals necessarily have compatible diagonals. In higher dimensions, the combinatorial type is not solely determined by the number of intersected edges. Thus, this construction cannot be generalized to higher dimensions (at least not in a direct way).
\end{example}
    \begin{figure}[H]
        \centering
        \begin{subfigure}[b]{0.4\textwidth}
            \centering
            \includegraphics[width=.8\textwidth]{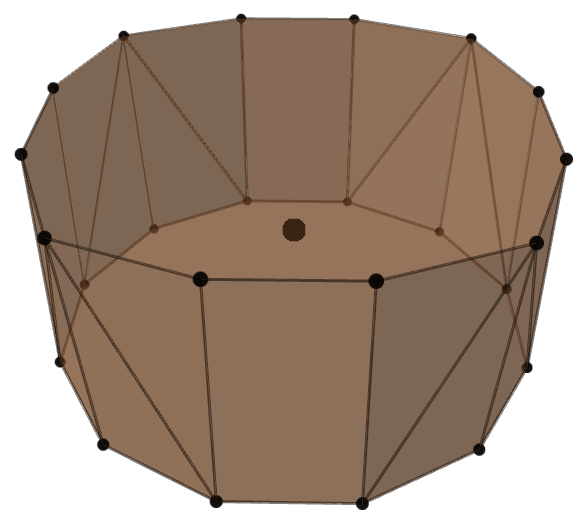}
        \end{subfigure}
        \begin{subfigure}[b]{0.4\textwidth}
            \centering
            \includegraphics[width=.8\textwidth]{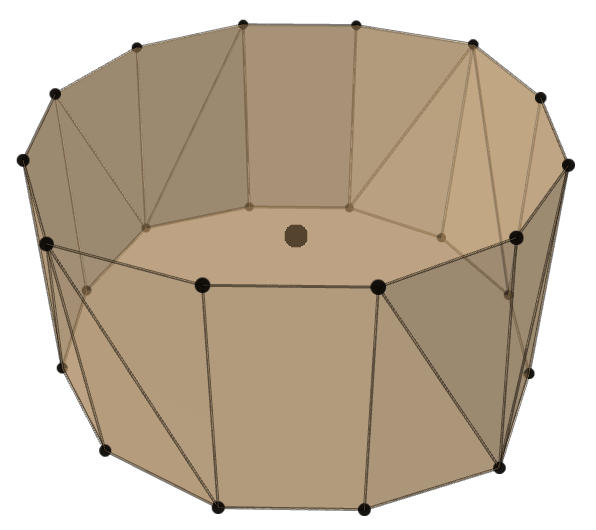}
        \end{subfigure}
        \caption{The drum polytopes $P$ and $Q$ (\Cref{ex:drums}).}
        \label{fig:drums}
    \end{figure}

\begin{example}[Cube variants]\label{ex:pointwise_cse}
    Consider the following variant of the cube $[-1,1]^3$:
    $$
        \begin{aligned}
        &v_1=\sma 1\\1\\-1\strix,\quad
        v_2=\sma 1\\-1\\-1\strix,\quad
        v_3=\sma -1\\1\\-1\strix,\quad
        v_4=\sma -1\\-1\\-1\strix,\quad \\[.2ex]
        v_5&=\frac{3}{2} \sma 1\\1\\0\strix,\quad
        v_6=\frac54\sma 1\\-1\\0\strix,\quad 
        v_7=\frac54\sma -1\\1\\0\strix,\quad
        v_8=\frac54\sma -1\\-1\\0\strix,\quad \\[.2ex]
    v_9&=\frac67\sma 1\\1\\1\strix,\quad
        v_{10}=\sma 1\\-1\\1\strix,\quad
        v_{11}=\sma -1\\1\\1\strix,\quad
        v_{12}=\frac65 \sma -1\\-1\\1\strix,
        \end{aligned}
        $$
    and $R = \conv(v_1,\dots,v_{12})$. We modify $R$ in two distinct ways. Let $\lambda = \frac{9}{10}$ and define
    \[
        P = \conv(v_1,\dots,v_{11},\lambda v_{12}), \qquad Q = \conv(\lambda v_1, v_2,\dots,v_{12}) \,
    \]
    see \Cref{fig:pointwise_cse}.
    By construction, the hyperplane arrangements $\mathcal H(R) = \mathcal H (P) = \mathcal H(Q)$ coincide. However, they have distinct combinatorial types, as can be seen as follows.
    $R$ is the convex hull of $3$ quadrilaterals $\conv(v_1,\dots,v_4), \ \conv(v_5,\dots,v_8)$, and $ \conv(v_9,\dots,v_{12})$.
    Combinatorially, $P$ is obtained from $R$ by ``breaking'' the upper facet into two triangles, while $Q$ is obtained by ``breaking'' the lower facet. This creates new edges, namely $v_{10}v_{11}$ in $P$, and $v_2v_3$ in $Q$. 
    As a result, $P$ has a triangular face (on the top) whose vertices all have degree $5$, but $Q$ has no such triangular face.
    Note that the new edges are antipodal, i.e., a central hyperplane $u^\perp$ intersects $v_2v_3$ if and only if it intersects $v_{10} v_{11}$. It follows that $P \cap u^\perp \simeq Q \cap u^\perp$ for every $u \in \Sph^{d-1}$, but $P \cap u^\perp \not \simeq R \cap u^\perp \not \simeq Q \cap u^\perp$ whenever $u^\perp$ intersects these edges.
    Similar to \Cref{ex:drums}, the main obstruction that guarantees combinatorial equivalence of slices, but not of the polytopes is to break two opposite faces into two triangles. 
    Any central slice intersecting the top edge in $P$ intersects the bottom edge in $Q$, and vice versa. Similar to \Cref{ex:drums}, this mechanism does not transfer to a centrally symmetric construction or to a construction in higher dimensions (at least not in a direct way).
\end{example}
\begin{figure}[H]
    \centering
    \hspace{1em}%
    \begin{minipage}{0.32\textwidth}
            \includegraphics[width=.74\textwidth]{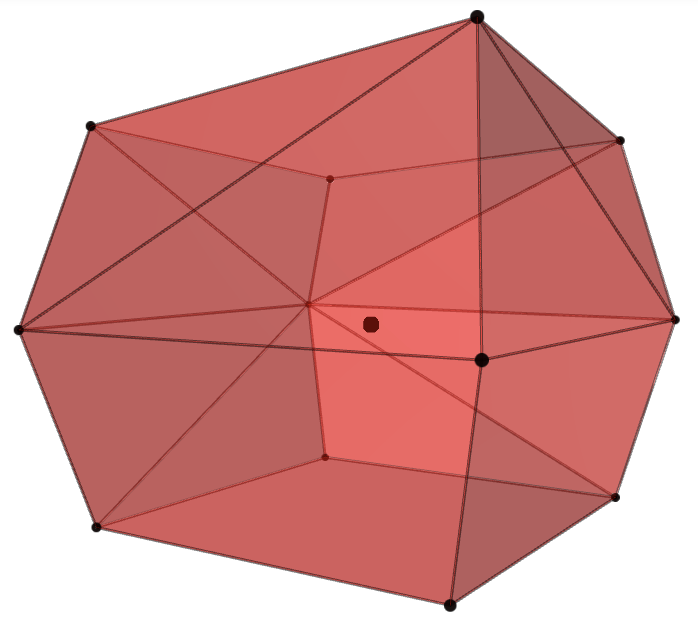}
        \end{minipage}
        \hfill%
        \begin{minipage}{0.32\textwidth}
            \includegraphics[width=.74\textwidth]{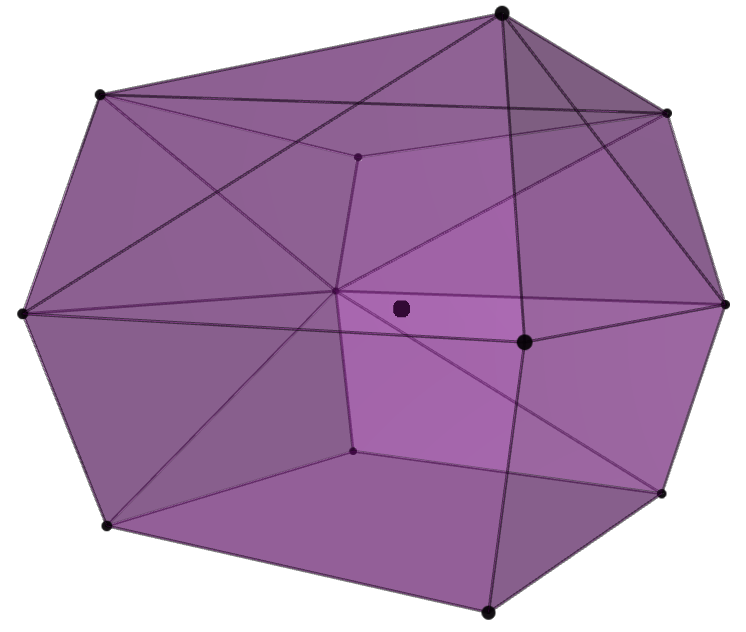}
        \end{minipage}
        \hfill%
        \begin{minipage}{0.32\textwidth}
        \includegraphics[width=.76\textwidth]{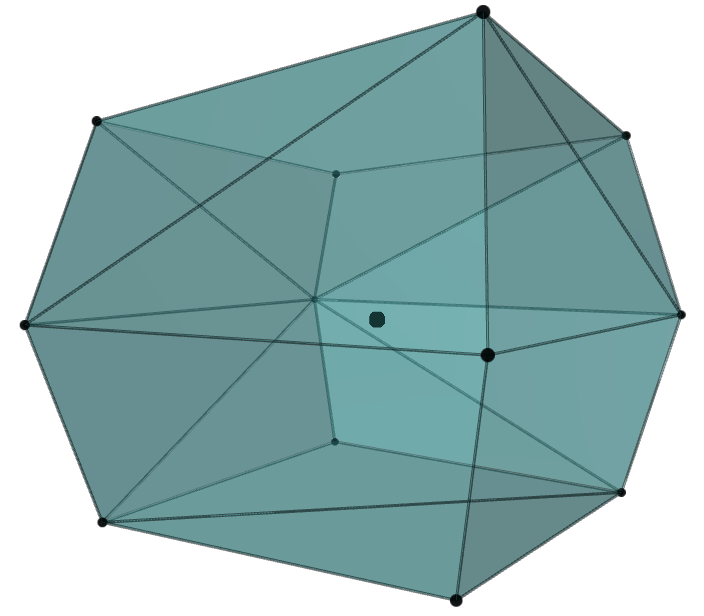}
        \end{minipage}
        \hspace{1em}
    \caption{Polytopes $R,P,Q$ (from left to right) with identical central hyperplane arrangement $\Hcal(\cdot)$, distinct combinatorial types, and where $P$ and $Q$ are pointwise c.s.e., but not to $R$ (\Cref{ex:pointwise_cse}).}
    \label{fig:pointwise_cse}
\end{figure}

\subsection{Hierarchy for central slices} \label{sec:hierarchy-central}

We now study which combinatorial and geometric properties of a polytope determine its central slices up to the following two definitions of ``same slices''.

\begin{definition}[central-slice-equivalence]
    Let $P,Q \subset \R^d$ be polytopes containing the origin in their interior. $P$ and $Q$
        are \emph{pointwise central-slice-equivalent (pointwise c.s.e)}
        if for every $u \in \Sph^{d-1}$ we have $P \cap u^\perp \simeq Q \cap u^\perp$.
        $P$ and $Q$ are \emph{setwise central-slice-equivalent (setwise c.s.e)} if the set of combinatorial types of central slices of $P$ agrees with that of $Q$.
\end{definition}

What does pointwise, or setwise, central-slice-equivalence say about polytopes $P,Q$, and conversely, what do polytopes $P,Q$ need to satisfy to be pointwise, or setwise, central-slice-equivalent?
Recall from \Cref{sec:reconstruction-central} that the combinatorics of central slices are governed by relatively open cells of the arrangement $\Hcal(P) = \{v^\perp \mid v \in \verts(P) \text{ and $v$ is not the origin}\}$.
However, if $v,w \in \verts(P) $ satisfy $w = \alpha v$ for some $\alpha <0$, then $v^\perp = w^\perp$ define the same hyperplane in $\Hcal(P)$. To retain this information, we also consider the \textit{signed} hyperplane arrangement $\Hcaltilde(P)$,
where each hyperplane $H_v := v^\perp$ comes with a positive side $H_v^+ := \{x \in \R^d \mid \langle v,x \rangle > 0\}$, and a negative side $H_v^- := \{x \in \R^d \mid \langle v,x \rangle < 0\}$. In particular, if $w = \alpha v$ with $\alpha <0$, then the hyperplanes $H_w$ and $H_{v}$ are distinct signed hyperplanes in the arrangement $\widetilde{\Hcal}(P)$.

We fix an ordering $v_1,\dots,v_n$ of the vertices of $P$. Then every relatively open cell $C$ of $\widetilde{\mathcal{H}}(P)$ can be labeled by the \emph{signed covector} $(\sgn(\langle u, v_i \rangle))_{i \in [n]} \in \{-,0,+\}^n$, which is independent of the choice of $u \in C$. The signed covector records on which side the cell $C$ is for each of the hyperplanes $v_i^\perp$, or, equivalently, on which side the vertex $v_i$ is for each of the hyperplanes $u^\perp$, for $u \in C$. We identify the \emph{oriented matroid} $\OM(\widetilde{\Hcal}(P))$ with the set of all signed covectors of the arrangement. 

If $P\simeq Q$ are combinatorially equivalent via an isomorphism $i:\Lcal(P) \to \Lcal(Q)$, then this induces a bijection $i^*:\Hcaltilde(P) \to \Hcaltilde(Q)$ given by $i^*(v^\perp) :=i(v)^\perp$, where the positive side of the hyperplanes is determined by $v \in \verts(P)$, and $i(v) \in \verts(Q)$, respectively.

\begin{definition}\label{def:LPinducingHP}
    Let $P \simeq Q$ be polytopes containing the origin in the interior. Let $i:\Lcal(P) \to \Lcal(Q)$ be an isomorphism of their face lattices, and  $i^*:\Hcaltilde(P) \to \Hcaltilde(Q)$ the corresponding bijection on the signed hyperplane arrangements.
    If $v^\perp = i^*(v^\perp)$ agree as signed hyperplanes for every $v \in \verts(P)$, then we say that $P \simeq Q$ \emph{induces (the equality)} $\Hcaltilde(P)=\Hcaltilde(Q)$. Similarly, if the oriented matroids of the signed hyperplane arrangements agree under the bijection $i^*$, then we say that $P \simeq Q$ \emph{induces (the equality)} $\OM(\Hcaltilde(P)) = \OM(\Hcaltilde(Q))$. 
\end{definition}

We will show that such a combinatorially induced equality of hyperplane arrangements will be sufficient to provide an equality of central slices. The following example shows that merely a combinatorial equivalence $P \simeq Q$ and an equality $\widetilde{\Hcal}(P) = \widetilde{\Hcal}(Q)$ is not sufficient, if the equality is not induced by the combinatorial isomorphism.

\begin{example}\label{ex:comb_equiv_and_H_not_induced2}
\label{ex:comb_equiv_not_ase}
    We construct two (centrally symmetric) polytopes obtained from the bipyramid over two different hexagons, by adding the antipodal pair of points $\pm (2/3, -2/3,1)$. Let
    \[
    P:= \conv \left\{  
        \pm \sma 1 \\ 0 \\ 0\strix, 
        \pm \sma 0 \\ 1 \\ 0\strix, 
        \pm \sma 1 \\ 1 \\ 0\strix, 
        \pm \sma 0 \\ 0 \\ 2\strix, 
        \pm \sma 2/3 \\ -2/3 \\ 1\strix\right\}
    \]
    and 
    \[
    Q := \conv \left\{  
        \pm \sma 3/2 \\ 0 \\ 0\strix, 
        \pm \sma 0 \\ 1 \\ 0\strix, 
        \pm \sma 1 \\ 1 \\ 0\strix, 
        \pm \sma 0 \\ 0 \\ 2\strix, 
        \pm \sma 2/3 \\ -2/3 \\ 1\strix\right\}.
    \]
    The base hexagon of $Q$ is different than that of $P$, by pushing out the vertices $\pm(1,0,0)$ to $\pm(3/2,0,0)$. This makes the extra points $\pm(2/3,-2/3,1)$ create different edge incidences at these vertices respectively. It yields a combinatorial isomorphism between $P$ and $Q$ which sends the apex $(0,0,2)$ in $P$ to $(3/2, 0,0)$ in $Q$. Moreover, $\widetilde{\mathcal H}(P) = \widetilde{\mathcal H}(Q)$. However, this equality is not \emph{induced} by the combinatorial isomorphism. To see this, note that the extra vertex $(2/3, -2/3, 1)$ has degree $5$ in $P$ and degree $4$ in $Q$. It can be checked that the slices of $P$ are $n$-gons up to $n=12$, whereas $Q$ also has a $14$-gon as a slice.
\end{example}
     \begin{figure}[H]
    \centering
    \hspace{8em}%
    \begin{minipage}{0.4\textwidth}
            \includegraphics[width=.6\textwidth]{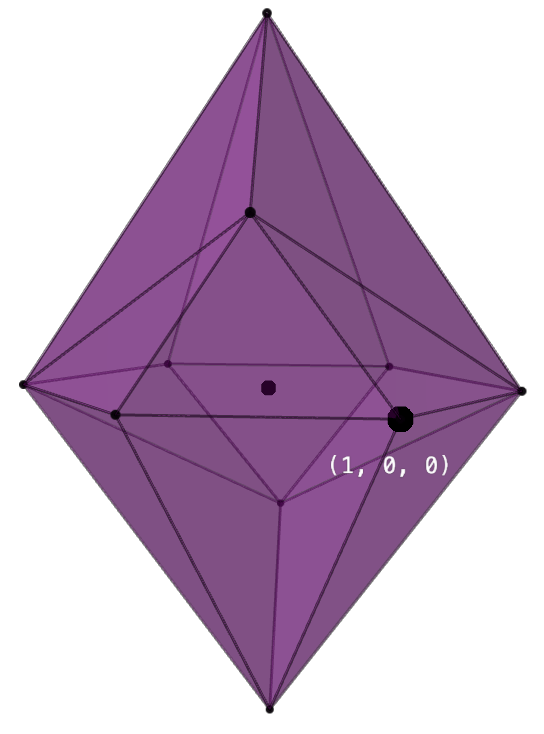}
        \end{minipage}
        \hfill%
        \begin{minipage}{0.4\textwidth}
            \includegraphics[width=.7\textwidth]{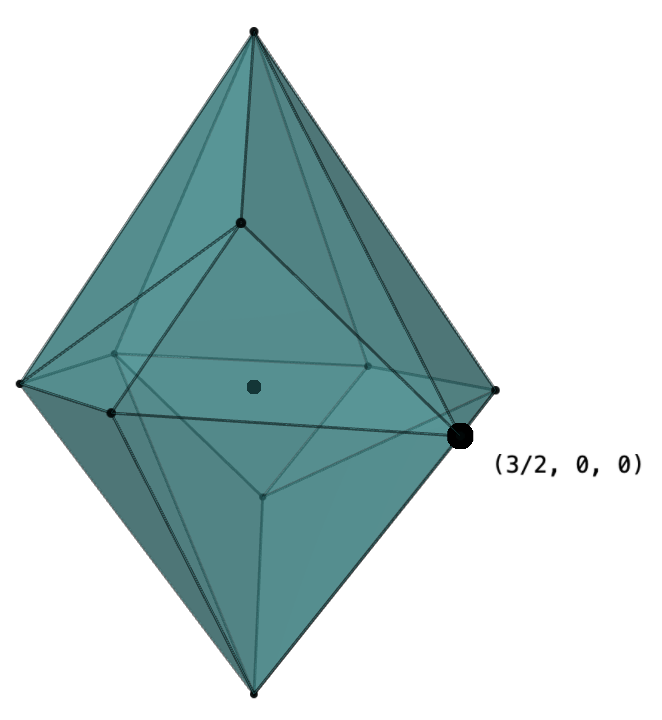}
        \end{minipage}
        \hspace{8em}
    \caption{Combinatorially equivalent polytopes $P,Q$ with $P \simeq Q$ not inducing  $\widetilde{\Hcal}(P)=\widetilde{\Hcal}(Q)$.}
\end{figure}

\begin{theorem}\label{thm:hierarchy-central}
    Let $d \geq 3$, and let $P,Q \subset \R^d$ be $d$-dimensional polytopes containing the origin in their interior.
    The following implications hold for every $d\ge3$. The counterexamples for the reverse non-implications are constructed in dimension $3$.
    \begin{center}
    \begin{tikzpicture}[
        every node/.style={align=center},
        bigprop/.style={align=center},
        prop/.style={text width=4.5cm, align=center}
    ]
    
    \node[bigprop] at (0,0) {$P \simeq Q$ inducing $\Hcaltilde (P) = \Hcaltilde (Q)$};
    
    \node at (-1.65,-0.65) {\rotatebox{-50}{$\Downarrow$}};
    \node at (-1.3,-0.65) {\rotatebox{-50}{$\not\Uparrow$}};
    \node at (1.3,-0.65) {\rotatebox{50}{$\Downarrow$}};
    \node at (1.65,-.65) {\rotatebox{50}{$\not\Uparrow$}};
    
    \node[prop] at (-2.6,-1.45) {$P\simeq Q$ inducing \\ $\OM(\Hcaltilde(P))=\OM(\Hcaltilde(Q))$};
    \node[prop] at (2.6,-1.45) {$P,Q$ pointwise c.s.e.};
    
    \node at (-3.7,-2.4) {\rotatebox{-25}{$\Downarrow $}};
    \node at (-3.4,-2.4) {\rotatebox{-25}{$\not\Uparrow$}};
    \node at (-1.85,-2.4) {\rotatebox{25}{$\Downarrow$}};
    \node at (-1.5,-2.4) {\rotatebox{25}{$\not\Uparrow$}};
    \node at (1.5,-2.4) {\rotatebox{-25}{$\Downarrow$}};
    \node at (1.85,-2.4) {\rotatebox{-25}{$\not\Uparrow$}};
    \node at (3.55,-2.4) {\rotatebox{35}{$\Downarrow$}};
    \node at (3.9,-2.4) {\rotatebox{35}{$\not\Uparrow$}};
    
    \node[prop] at (-4,-3.) {$P \simeq Q$};   
    \node[prop] at (-0,-3.) {$P,Q$ setwise c.s.e.};
    \node[prop] at (4,-3.) {$\mathcal H (P) = \mathcal H (Q)$};
    
    \end{tikzpicture}
    \end{center}
    When no arrows are depicted, then
    neither direction of implication holds $(\substack{\Large \not \Rightarrow \\ \Large \not \Leftarrow})$.
\end{theorem}

\begin{proof}
    First we show the implications that hold true.
    Clearly $\Hcaltilde(P)=\Hcaltilde(Q)$ implies $\OM(\Hcaltilde(P))=
    \OM(\Hcaltilde(Q))$.
    Suppose $P\simeq Q$ via a combinatorial isomorphism $i$, that induces $\Hcaltilde(P)=\Hcaltilde(Q)$. Let $u^\perp$ be some central slice for $u \in \Sph^{d-1}$, and let
    $\sigma(P,u) = (\sgn v^\top u )_{v \in \verts(P)} \in \{ +,-,0 \}^{\verts(P)}$ be the sign vector of how $u$ separates vertices of $P$.
    Then $\sigma(Q,u) := (\sgn i(v)^\top u )_{v \in \verts(P)} = \sigma(P,u)$, since $i$ just scales each vertex of $P$ with a non-zero scalar. Hence $u^\perp$ intersects the same set of edges of $P$ and $Q$ (under the combinatorial isomorphism). By \Cref{prop:face-lattice-of-slice} we have $P \cap u^\perp \simeq Q \cap u^\perp$, hence $P$ and $Q$ are pointwise c.s.e. Similarly, suppose that the oriented matroids of $\Hcaltilde(P)$ and $\Hcaltilde(Q)$ agree. Let $u \in \Sph^{d-1}$, then the sign vector $\sigma(P,u)$
    is a covector of $\OM(\Hcaltilde(P))$, and corresponds to a non-empty, relatively open cell of $\Hcal(P)$. Since the oriented matroids agree, the same covector corresponds to a non-empty relatively open cell of $\Hcal(Q)$, i.e., there exists a $u'$ in this cone with $\sigma(Q,u')=\sigma(P,u)$. In particular $u^\perp$ intersects the same set of edges in $P$ as $(u')^\perp$ does in $Q$ (under the combinatorial isomorphism).
    Clearly pointwise c.s.e. implies setwise c.s.e. That pointwise c.s.e. recovers the hyperplane arrangements follows from \Cref{prop:reconstruct-arrangement}.

    For the implications that are false: Let $P$ and $Q$ be two distinct triangles, then generically $\Hcaltilde(P) \neq \Hcaltilde(Q)$, but the oriented matroids of their signed hyperplane arrangements agree. \Cref{ex:drums} (drums) as well as \Cref{ex:pointwise_cse} give examples of $P,Q$ that are pointwise c.s.e., but not combinatorially equivalent. The pair of polytopes $R,P$ or $R,Q$ in \Cref{ex:pointwise_cse} have the same hyperplane arrangements, but they are not setwise c.s.e. (hence also not pointwise c.s.e.).
    The regular and perturbed bipyramid in \Cref{ex:perturbed_bipyramid} are combinatorially equivalent, but their central hyperplane arrangements have distinct oriented matroids, and they are also not setwise c.s.e.
    Lastly, \Cref{ex:0sym-setwise-cse-not-comb-equiv}, a bipyramid over a square with a pair of stacked faces, and the three-cube, with two pairs of ``broken faces'', both have central slices with only $4,6$ and $8$ vertices, hence they are setwise c.s.e., but they are not pointwise c.s.e., nor combinatorially equivalent, and they have distinct central hyperplane arrangements (with distinct oriented matroids).
\end{proof}

\begin{remark}
 Note that, in particular, all positive implications hold for centrally symmetric examples, since they hold for all polytopes. All counterexamples proving the non-implications are centrally symmetric, except for \Cref{ex:drums} and \Cref{ex:pointwise_cse}. Thus, in principle, pointwise central-slice-equivalence may still imply $P \simeq Q$ when we restrict to the class of centrally symmetric polytopes. \Cref{cor:cs-simple-polytope-reconstructible} shows this implication for simple centrally symmetric polytopes whose vertices are in symmetric 3-general linear position.
\end{remark}

\subsection{Hierarchy for affine slices} \label{sec:hierarchy-affine}

In parallel to \Cref{sec:hierarchy-central}, 
we study which combinatorial and geometric properties of a polytope $P$ determine its affine slices. As above, we consider two notions of affine-slice-equivalence.

\begin{definition}[affine-slice-equivalence]
    Let $P,Q \subset \R^d$ be polytopes. $P$ and $Q$
        are \emph{pointwise affine-slice-equivalent (pointwise a.s.e)}
        if for every affine hyperplane $H$ one has $P \cap H \simeq Q \cap H$.
        $P$ and $Q$
        are \emph{setwise affine-slice-equivalent (setwise a.s.e)}
    if the set of combinatorial types of affine slices of $P$ agrees with that of $Q$.
\end{definition}

\begin{proposition}\label{prop:ase-properties}
Let $d \geq 3$, and let $P,Q \subset \R^d$ be $d$-dimensional polytopes.
Two polytopes are pointwise affine-slice-equivalent if and only if they are equal. Moreover, for every $d \geq 3$, combinatorially equivalent polytopes need not be setwise affine-slice-equivalent. Conversely, if $d=3$, setwise affine-slice-equivalent polytopes can be combinatorially distinct and can even have different $f$-vectors.
\end{proposition}

\begin{proof}
    For the first statement, note that if $P,Q$ are pointwise affine-slice-equivalent, then $P \cap H = \emptyset$ if and only if $Q \cap H = \emptyset$. Thus, $\R^d \setminus P = \R^d \setminus Q$, and hence $P = Q$.
    For every $d \geq 3$, \Cref{ex:perturbed_bipyramid} gives combinatorially equivalent polytopes that are not setwise affine-slice-equivalent. Conversely, \Cref{ex:0sym-setwise-cse-not-comb-equiv} gives, in dimension $3$, setwise affine-slice-equivalent polytopes that are combinatorially distinct and have different $f$-vectors.
\end{proof}

Recall from \Cref{sec:lattices} that the combinatorial type of the slice is uniquely determined by the set of edges of $P$ which are intersected by the affine hyperplane $H_u(\beta) = \{x \mid \langle x,u\rangle + \beta = 0\}$, defining the slice $P \cap H_u(\beta) = \{ x \in P \mid \langle (x,1), (u,\beta)\rangle = 0\}$. It follows, e.g., from \cite{Brandenburg2025} that the set of hyperplanes intersecting the same sets of edges are characterized by those normals $(u,\beta)$ contained in the same relatively open cell of the hyperplane arrangement $\mathcal H(P \times \{1\}) = \{(v,1)^\perp \mid v \text{ vertex of } P\}$. In contrast to the last section, $\Hcaltilde(P \times \{1\})$ can be determined from $\Hcal(P \times \{1\})$ since no two vertices of $P \times \{1\}$ define the same hyperplane, and each hyperplane is oriented such that $e_{d+1}$ is on its positive side, and $-e_{d+1}$ on its negative side.

Recall that the \emph{oriented matroid of a polytope} $P$, denoted $\OM(P)$, is the oriented matroid of the affine point configuration $\verts(P)$, and that the oriented matroid determines the combinatorial type, but not vice versa. We prove the following.

\begin{theorem}\label{th:affine-hierarchy} Let $P,Q \subset \R^d$ be $d$-dimensional polytopes. The following implications and non-implications hold for every $d\ge3$, unless stated otherwise.
\begin{center}
    \centering
    $P = Q$ $\iff$
    $P,Q$ pointwise a.s.e
    $\iff$
    $\mathcal H(P \times \{1\}) = \mathcal H(Q \times \{1\})$ \\
    $\Downarrow \ \not\Uparrow$ \\
    $\OM(P) = \OM(Q)$ $\iff$
    $ \OM (\mathcal H(P \times \{1\})) = \OM(\mathcal H(Q \times \{1\}))$ \\
    $\Downarrow \ \not\Uparrow$ \hspace{15em}  $\Downarrow \ \not\Uparrow$ \\
    $P, Q$ 
    setwise a.s.e 
    \hspace{4em} $\Large\substack{\Large \not \Rightarrow \text{ if } d=3\\ \hspace{-2.4em} \Large \not \Leftarrow}$ \hspace{1.5em} $P\simeq Q$ \hspace*{1em}
    \end{center}
\end{theorem}
    
\begin{proof}
    \Cref{prop:ase-properties} implies the equivalence between $P = Q$ and
    $P,Q$ being pointwise a.s.e. If $P=Q$, then $\mathcal H(P \times \{1\}) = \mathcal H(Q \times \{1\})$. To show the converse, note that every hyperplane in $\mathcal H(P \times \{1\})$ has normal vector with non-zero last coordinate, and that no two vertices of $P \times \{1\}$ can define the same hyperplane. We can thus recover the vertices of $P$ uniquely by the normal vectors of each hyperplane scaled such that their last coordinates are $1$.
    If $P = Q$, then they have the same oriented matroid. For an example illustrating that the converse does not hold, consider, e.g., intervals $P = \conv(0,1), Q = \conv(0,2) \subset \R^1$. A counterexample in general dimension is given by any simplicial polytope $P$, where $Q$ is a generic perturbation of its vertices. For the equivalence in the second line: the oriented matroid of $\mathcal H(P \times \{1\})$ is the oriented matroid of the vector configuration $\{(v,1) \colon v \in \verts(P) \}$. Linear dependencies of $\{ (v,1) \colon v \in \verts(P) \}$ are in one-to-one correspondence with affine dependencies of the points $\{v \colon v \in \verts P\}$, hence the two oriented matroids agree.
    If $P,Q$ have the same oriented matroid, then, in particular, their covectors agree. The covectors are sign vectors, encoding the possible separations of vertices by hyperplanes into sides $\{+,-,0\}$. On the one hand, this recovers the combinatorial type of the polytope. On the other hand, it encodes the linear separations of the vertices of the polytope. These linear separations determine, combinatorially, which edges can be intersected simultaneously by a hyperplane. If $P,Q$ have the same oriented matroid, the set of covectors determines the combinatorial types of their slices, so $P,Q$ are setwise affine-slice-equivalent. 
    By \Cref{prop:ase-properties}, if $P,Q$ are setwise affine-slice-equivalent, then this does not imply that $P,Q$ are combinatorially equivalent. In particular, $P,Q$ do not have to have the same oriented matroid. The bipyramid and perturbed bipyramid in \Cref{ex:perturbed_bipyramid} are an example of polytopes with the same combinatorial type but distinct oriented matroids.
\end{proof}

\section{Combinatorial versions of classical slicing problems}\label{sec:slicing-problems}

In this section we study whether inequalities between the face numbers of slices imply inequalities between the face numbers of the original polytopes. We formulate this question in analogy with the Busemann-Petty problem and Bourgain's slicing problem. Our main result is negative: for general polytopes, these combinatorial analogues fail not only in their exact form, but also up to constants depending on the dimension. The failure persists for both central and affine sections.
We first identify classes of polytopes for which the combinatorial Busemann-Petty statement (\Cref{problems}\labelcref{prob:BP})~holds.

\begin{proposition}\label{prop:affirmative-BP-pair}
    Let $P,Q \subset \R^d$ be $d$-dimensional polytopes containing the origin in their interiors, let $k \in [d-1]$ and assume that $f_{k-1}(Q \cap u^\perp) \leq f_{k-1}(P \cap u^\perp)$ holds for all $u \in \Sph^{d-1}$. If there exists a central hyperplane which properly intersects all $k$-dimensional faces of $Q$ then $f_{k}(Q) \leq f_k(P)$. 
\end{proposition}

\begin{proof}
    Let $u \in \Sph^{d-1}$ such that $u^\perp$ properly intersects all $k$-dimensional faces of $Q$. In particular $u^\perp$ cannot contain any $(k-1)$-dimensional face of $Q$. Then, by \Cref{lem:types-of-faces-in-slices}  we have $f_{k-1}(Q \cap u^\perp) = f_k(Q)$.
    On the other hand, \Cref{lemma:fkmin1_of_slice_leq_fk_of_poly} gives $f_{k-1}(P \cap u^\perp) \leq f_k(P)$.
\end{proof}

We point out that the class of $Q$ satisfying the assumption of \Cref{prop:affirmative-BP-pair} is non-empty. For $k=d-1$, examples are the $d$-dimensional cube $[-1,1]^d $ (where all central hyperplanes which do not separate two opposite facets intersect all facets) or pyramids (where a hyperplane which cuts off an edge incident to the apex satisfies the assumptions).

\begin{remark}
    Khovanskii \cite[Corollary~4.5]{Khovanskii06:SectionsPolytopes} shows that for a simple $d$-dimensional polytope $P$ and a generic hyperplane $H$ there exists a $\lfloor d/2 \rfloor$-face of $P$ that $H$ does not intersect.  That means, by \Cref{lem:upward-downward-closure}, that for all $k \leq \lfloor d/2 \rfloor$ there does not exist a hyperplane $H$ (generic or non-generic) that intersects all $k$-faces of $P$ in their relative interior. In view of this, Khovanskii calls a generic slice of a simple $d$-polytope \textit{successful} if it intersects all $\lfloor d/2\rfloor +1$-faces in their relative interior. (By \Cref{lem:upward-downward-closure} this is equivalent to intersecting all $>d/2$-dimensional faces in their relative interior.)
\end{remark}

We now turn to constructing counterexamples to each of the problems in \Cref{problems}. We do this by constructing two sequences of polytopes, a sequence of polytopes with few faces, whose slices have relatively many faces (a \emph{big-slice sequence}), and a sequence of polytopes with many faces, whose slices have relatively few faces (a \emph{small-slice sequence}). Any such sequence will provide a counterexample. We then provide an explicit small-slice sequence in dimension $2$, and give a construction to obtain a small-slice sequence of dimension $d$ from a small-slice sequence of dimension $d-1$ (the \emph{globe construction}). For the big-slice sequence, we provide a sequence of polytopes in dimension $3$, and show that bipyramids over a big-slice sequence in dimension $d-1$ yield a big-slice sequence in dimension $d$. 

Recall that for two functions $f,g:\N \to \N$ we write $f \in \Ocal(g)$, if there exists a constant $c>0$, and an $N_0 \in \N$ such that $f(n) \leq c\cdot g(n)$ for all $n \geq N_0$. Analogously, we write $f \in \Omega(g)$ if there exists a constant $c>0$, and $N_0 \in \N$ such that $f(n) \geq c \cdot g(n)$ for all $n \geq N_0$. Then, $ \Theta(g) = \Ocal(g) \cap \Omega(g)$. 

We define our big-slice sequence and small-slice sequence at the level of generality that our upcoming constructions satisfy.

\begin{definition}[big-slice sequence and small-slice sequence]
    Let $(P_n)_{n \in \N}$ be a sequence of $d$-dimensional polytopes in $\R^d$. 
    We call $(P_n)_{n \in \N}$ a \emph{big-slice sequence} if the following conditions hold:
    \begin{enumerate}[label=(\roman*)]
        \item For every $0 \leq k \leq d-1$ we have $f_k(P_n) \in \Theta(n)$, and
        \item for every $0 \leq k \leq d-2$ we have $\displaystyle \min_{u \in \Sph^{d-1}} f_k(P_n \cap u^\perp) \in \Omega(n)$.
    \end{enumerate}
    We call $(P_n)_{n \in \N}$ a \emph{small-slice sequence} if 
    \begin{enumerate}[label=(\roman*)]
        \item for every $0 \leq k \leq d-1$ we have $f_k(P_n) \in \Theta(n^{d-1})$, and
        \item for $0 \leq k \leq d-2$ we have $\displaystyle \max_{\substack{H \subset \R^d \\ \text{aff. hyperplane}}} f_k(P_n \cap H) \in \mathcal{O}(n^{d-2})$.
    \end{enumerate}
\end{definition}

We now show that a pair of small-slice and big-slice sequences are enough to disprove \Cref{problems}. This serves as the blueprint for what follows in this section.

\begin{theorem}[Slice-sequences are counterexamples to combinatorial Bourgain]
    For every dimension $d \geq 3$, every $k \in [d-1]$, for every constant $C > 0$, every big-slice sequence $(P_n)_{n \in \N}$, and every small-slice sequence $(Q_m)_{m \in \N}$, there exist $n,m \in \N$ such that $f_{k-1}(Q_m \cap u^\perp) \leq f_{k-1}(P_n \cap u^\perp)$ for every $u \in \Sph^{d-1}$, and $f_k(Q_m) > C \cdot f_k(P_n)$.
    \label{prop:bp_counterexample}
\end{theorem}

\begin{proof}
    By definition, there exist constants $C_1, C_2, C_1', C_2'$ such that for all central hyperplanes $u^\perp$, for all $k \in [d-1]$ and for sufficiently large $n,m \in \N$ we have
    \begin{center}\vspace{-1.5em}
    \begin{minipage}{0.4\textwidth}
        \begin{align*}
        f_k(P_n) &\leq C_1 n \\
        f_k(Q_m) &\geq C_2 m^{d-1}
    \end{align*}
    \end{minipage}
    \begin{minipage}{0.4\textwidth}
        \begin{align*}
        f_{k-1}(P_n \cap u^\perp) &\geq C_1' n \\
        f_{k-1}(Q_m \cap u^\perp) &\leq C_2' m^{d-2} \ .
    \end{align*}
    \end{minipage}
    \end{center}

    In order to find $P_n$ and $Q_m$ where  $C \cdot f_k(P_n) < f_k(Q_m)$ and $f_{k-1}(P_n \cap u^\perp) \geq f_{k-1}(Q_m \cap u^\perp)$ for all central hyperplanes $u^\perp$, it suffices to choose $n, m \in \N$ such that $C C_1 n < C_2 m^{d-1}$ and $C_2' m^{d-2} \leq C_1' n$. This is equivalent to 
    \[ \frac{C_2'}{C_1'} m^{d-2} \leq n < \frac{C_2}{C C_1} m^{d-1} \ . \]
    Since $d \geq 3$, the expression on the right-hand side grows asymptotically faster in $m$ than the expression on the left-hand side. Hence, for sufficiently large $m$ there exists a choice of $n$ satisfying these constraints.
\end{proof}

It thus suffices to construct a pair of such sequences.
In the following, we show that for a big-slice sequence $(P_n)_{n \in \N}$ and a small-slice sequence $(Q_m)_{m \in \N}$ of $d$-dimensional polytopes, the sequences $(\bipyr (P_n))_{n \in \N}$ and $(\globe(m, Q_m))_{m \in \N}$ are, respectively, a big-slice sequence and a small-slice sequence of $(d+1)$-dimensional polytopes. We will also give explicit examples of a big-slice sequence and a small-slice sequence in $\R^3$. Combining this with \Cref{prop:bp_counterexample} and induction yields a counterexample to the combinatorial versions of the Busemann-Petty problem and Bourgain's slicing problem in every dimension $d \geq 3$.

\subsection{Small-slice sequences and globes}\label{sec:small-slice}

To make the following definition more intuitive, note that for $R \in (-1,1)$ defining the affine hyperplane
$H = \{x \in \R^{d+1} \mid x_{d+1} = R\}$, the slice $\Sph^d \cap H$ of the $d$-dimensional unit sphere is a $(d-1)$-dimensional sphere of radius $\sqrt{1-R^2}$.

\begin{definition}[Globes]\label{def:globes}
    For a polytope $P \subset \R^d$ with the origin in its interior, and $n \geq 2$, we define the \emph{$n$-globe over $P$} as 
    \[ \globe(n,P) := \conv\left( \sqrt{1-\left(\frac{k}{n}\right)^2} \cdot P \times \Bigl\{\frac{k}{n}\Bigr\} \mid k \in \Z, -n+1 \leq k \leq n-1\right) \subset \R^{d+1} \ . \]

    For $P_n \subseteq \R^2$, the regular $n$-gon with vertices on the unit circle, we call
    \[ G^d(n) := \underbrace{\globe(n, \globe(\ldots (n, \globe(n, P_n)) \ldots ))}_{d-2 \text{ times}} \subseteq \R^d \]
    the \emph{$d$-dimensional globe}, i.e., it is obtained by taking $d-2$ iterative $n$-globes over the regular $n$-gon.
\end{definition}

In \Cref{fig:3-globe} we depict the three-dimensional globe over the $14$-gon, $G^3(14)$. We begin with examining the facial structure of a globe.
\begin{figure}[ht!]
    \centering
    \includegraphics[width=0.3\linewidth]{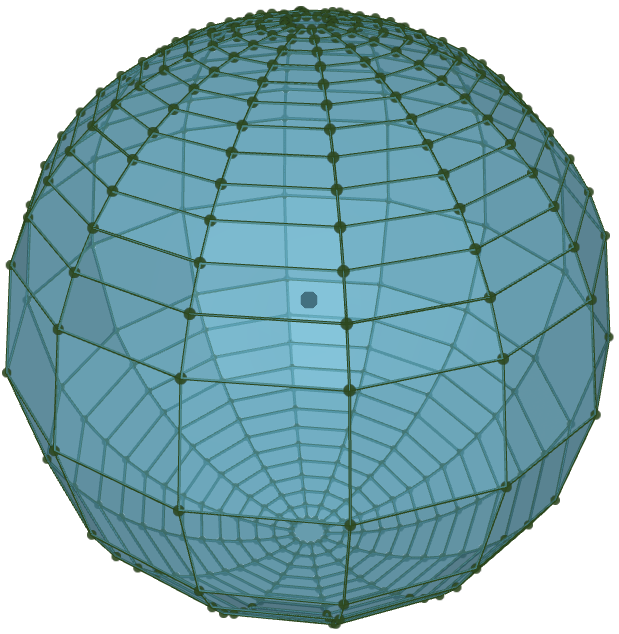}
    \caption{Three-dimensional globe over a regular $14$-gon.}
    \label{fig:3-globe}
\end{figure}

\begin{proposition}\label{prop:fvector_globe}
    Let $P \subset \R^d$ be a polytope containing the origin in its interior.
    The faces of $\globe(n,P)$ are of the following form: 
    \begin{enumerate}[label=\textup{(}\roman*\textup{)}]
        \item (faces at the ``poles'') $P_{n-1} = \sqrt{1 - (\frac{n-1}{n})^2}P \times \{\frac{n-1}{n}\}, \ P_{-n+1} = \sqrt{1 - (\frac{n-1}{n})^2}P \times \{\frac{-n+1}{n}\}$
        \item (copies of proper faces of $P$) $F_i = \sqrt{1 - (\frac{i}{n})^2}F \times \{\frac{i}{n}\}$ for $i \in \Z$, $-n+1 \leq i \leq n-1$, and for all proper faces $F$ of $P$,
        \item (generalized prisms of proper faces of $P$) $F_{i,i+1} = \conv(F_i, F_{i+1})$ for $i \in \Z$, $-n+1 \leq i \leq n-2$, and for all proper faces $F$ of $P$.
    \end{enumerate} 
    
    The $f$-vector of $\globe(n,P) \subset \R^{d+1}$ is given by
    \[ f_k(\globe(n,P)) = \begin{cases}
        (2n-1)f_k(P)  & \text{if } k=0\\
        (2n-1)f_k(P) + (2n-2)f_{k-1}(P) & \text{if }  1 \leq k \leq d-1 \\
        (2n-2)f_{k-1}(P) + 2 & \text{if } k=d \ .
    \end{cases}\]
\end{proposition}

\begin{proof}
  The polytope $\globe(n,P)$ is the convex hull of $2n-1$ copies of $P$, which are embedded on parallel hyperplanes. The scaling of each of the copies ensures that all vertices are in convex position, so that $\globe(n,P)$ can be thought of as a collection of generalized prisms of $P$ on top of each other. 
  More precisely, let $H_u(\beta)$ define a supporting hyperplane of $\globe(n,P)$, and write $u = (a,\lambda)$. On the $i$-th copy of $P$, the maximum of the functional $\langle \cdot,u\rangle$ is
  The copies of $P$ on which the supporting hyperplane $H_u(\beta)$ attains its maximum are exactly those indices $i$ for which the above expression is maximal. If $a\neq 0$, then
$\max_{x\in P}\langle x,a\rangle>0$,
because $0\in\inter(P)$. Therefore, dividing the expression by $\max_{x\in P}\langle x,a\rangle>0$ leaves the set of maximizers invariant.
This modified expression is obtained by evaluating the strictly concave function
$$
\varphi_a : t\longmapsto \sqrt{1-t^2}
+
\frac{\lambda}{\max_{x\in P}\langle x,a\rangle}t
$$
at the points $t=\frac{i}{n}$, where $-n+1\leq i\leq n-1$. Hence, its maximum on the discrete set
$
\left\{\frac{i}{n}\mid -n+1\leq i\leq n-1\right\}
$
is attained either at a single index $i$, or at two consecutive indices $j$ and $j+1$. Thus every exposed face of $\globe(n,P)$ is supported either on one copy of $P$, or on two consecutive copies of $P$, i.e., at faces of the form $F_i$ or $F_{j,j+1}$. The case $a = 0$ yields the faces at the poles.

Next we show that indeed all $F_i,F_{i,i+1}$ arise as a face of $\globe(n,P)$. Let $a \in \R^{d}, \alpha \in \R$, and $F$ be a face of $P$ be such that $F = P \cap H_a(-\alpha) = \{x \in P \mid \langle a, x \rangle = \alpha\}$. Let $i \in \{-n+1, \ldots, n-2\}$ and define $r_i = {\textstyle{\sqrt{1-\lp\frac{i}{n}\rp^2}}}$. We consider $\lambda$ and $\beta$ such that with $u=(a,\lambda)$, the hyperplane $H_u(-\beta)$ is supporting for $\globe(n,P)$ at $F_{i,i+1}$. For this, let $\lambda = n\alpha(r_i-r_{i+1})$ and $\beta = \alpha (r_i + i(r_i - r_{i+1})) = \alpha r_i + i\tfrac{\lambda}{n}$. 
First, we show that $H_u(-\beta)$ is supporting for $\globe(n,P)$.
Write any point in $\globe(n,P)$ as $x_j := (r_jx,\frac{j}{n})$ for $x \in P$ and $j \in \{-n+1, \ldots, n-1\}$. Then since $H_a(-\alpha)$ is supporting for $P$, i.e., $\langle x,a\rangle \leq \alpha$ for all $x \in P$ we get 
\begin{align}\label{eq:xju_value}
    \langle x_j, u \rangle = r_j \langle x, a \rangle + \lambda \frac{j}{n}  = r_j \langle x,a\rangle + j\alpha (r_i - r_{i+1}) \leq \alpha (r_j + j(r_i - r_{i+1}))
\end{align}
Next, we show that $F_i, F_{i+1} \subset H_u(-\beta)$. For $x \in F$ and $j \in \{i, i+1\}$, indeed 
\begin{align*}
    \langle x_j, u \rangle = r_j \alpha + j\alpha (r_i - r_{i+1}) =  \alpha (r_i + i(r_i - r_{i+1})) = \alpha (r_{i+1} + (i+1)(r_i - r_{i+1})) = \langle x_{i+1}, u \rangle \ .
\end{align*} Thus, $F_i,F_{i+1} \subset H_u(-\beta)$, and so $F_{i,i+1}\subset H_u(-\beta)$. To show that $H_u(-\beta)$ is a supporting hyperplane of $\globe(n,P)$ with exposed face $F_{i,i+1}$, recall that $\varphi_a$, defined above, is a strictly concave function. We get $\varphi_a(\tfrac{i}{n})= \varphi_a(\tfrac{i+1}{n})= \tfrac{\beta}{\alpha}$.
By strict concavity, for all $\tfrac{j}{n}<\tfrac{i}{n}$ and for all $\tfrac{j}{n} > \tfrac{i+1}{n}$ it is $\varphi_a(\frac{i}{n}) > \varphi_a(\tfrac{j}{n})$. 
Moreover, since $0 \in \inter(P)$ and $\alpha = \max_{x \in P} \langle x, a \rangle$, we have $\alpha > 0$. 
Thus, for any $x_j \in P$, by \eqref{eq:xju_value} 
\begin{align*}
   \tfrac{\beta}{\alpha} = \varphi_a(\tfrac{i}{n}) \geq \varphi_a(\tfrac{j}{n}) = r_j + \tfrac{\lambda}{\alpha} \cdot \tfrac{j}{n}
    =r_j  + j(r_i - r_{i+1}) 
    \geq \tfrac{1}{\alpha} \langle x_j,u \rangle 
\end{align*}
and the first inequality is strict exactly when $j \not \in \{i,i+1\}$, the second inequality is strict exactly when $x \not \in F$.
This proves that $H_u(-\beta)$ is a supporting hyperplane for $\globe(n,P)$ at $F_{i,i+1}$. Finally, $F_i$ and $F_{i+1}$ are each faces of $F_{i,i+1}$, hence also a face of $\globe(n,P)$.

  By construction, the vertices of $\globe(n,P)$ are exactly the vertices of the individual copies of $P$, so $f_0(\globe(n,P)) = (2n-1)f_0(P)$.
  For each $k \in \left\{1, \ldots, d-1 \right\}$, any $k$-face is either a $k$-face $F_i$ for some $k$-face $F$ of $P$, or the convex hull $F_{j,j+1}$ of two copies of the same $(k-1)$-face $F$ of $P$. This yields $f_k(\globe(n,P)) = (2n-1)f_k(P) + (2n-2)f_{k-1}(P)$.
  The $d$-dimensional faces are of the form $F_{j,j+1}$ for facets $F$ of $P$, and the two copies of $P$ at the poles. Thus $f_d(\globe(n,P)) = (2n-2)f_{d-1}(P)+2$.
\end{proof}

We now show that the globe construction yields a small-slice sequence. For this, we need to understand how many $k$-faces can be intersected simultaneously by a single hyperplane, for every $k \in [d]$. The following technical lemma shows that for $\globe(n,P)$ and a fixed face $F \preceq P$, a hyperplane can only intersect a small number of associated faces $F_{i,i+1}$ without simultaneously intersecting the faces $F_i$ as well.

\begin{lemma}\label{lem:globe-intersections}
    Let \(P \subset \R^d\) be a polytope containing the origin in its interior, let \(F\) be a \(k\)-face of \(P\), $k\geq 0$, and let \(H \subset \R^{d+1}\) be a hyperplane defining the open halfspaces $H^+,H^-$. 
    Then the following statements hold for faces $F_i$ and $F_{i,i+1}$ of $\globe(n,P)$ (as defined in \Cref{prop:fvector_globe}).
    \begin{enumerate}[label=\textup{(}\arabic*\textup{)}]
    \item Let $i<j$. If $F_i,F_j \subset H^+$ and $F_{j,j+1} \cap H \neq \emptyset$, then $F_{l,l+1} \subset H^+$ for all $i \leq l \leq j-1$. \label{item:globe-intersections-step1}
    \item There are at most two indices \(i\) such that
        $\dim(F_{i,i+1} \cap H)=k$ and $F_i \cap H=\emptyset.$\label{item:globe-intersections-punchline}
    \end{enumerate}
\end{lemma}

\begin{proof} 
    To prove \labelcref{item:globe-intersections-step1}, let $v$ be a vertex of $F$, and denote by 
    $$\textstyle v_i = \left(\sqrt{1-(\frac{i}{n})^2}\ v, \frac{i}{n}\right)$$ 
    the corresponding vertex of $F_i$.
    By the construction of the globe, all vertices $v_{i}$, $i \in [-n+1,n-1]$ lie on a common curve of degree $2$, namely $\gamma^v :=  \{ (\alpha v, \mu) \mid \alpha,\mu \in \R, \ \alpha^2 + \mu^2 =1\}$. 
     We denote by $\gamma^v_{ij} \subset \gamma^v$ the arc with endpoints $v_i$ and $v_j$.
     We show that for all $v \in \verts(F)$ and $l = i, \ldots j$ we have $v_l \in H^+$. Then, convexity implies that $F_l \subset H^+$ for all $i \leq l \leq j$ and $F_{l,l+1} \subset H^+$ for all $i \leq l \leq j-1$, as desired.
     Let $H = H_{(u,\lambda)}(\beta) = \{x \in \R^{d+1} \mid \langle x, (u,\lambda) \rangle + \beta = 0\}$, and $H^- = \{x \mid \langle x,(u,\lambda)\rangle + \beta < 0\}$. By assumption, $F_{j,{j+1}} \cap H \neq \emptyset$ and $F_j \subset H^+$. Thus, there exists some $w \in \verts(F)$ such that $w_{j+1} \in \overline{H^-}$. Let $v^* \in \verts(F)$ be the vertex minimizing $\langle \cdot,u\rangle$ among all vertices of $F$. Then, in particular, $\langle v^*_{j+1}, (u,\lambda) \rangle + \beta \leq \langle w_{j+1}, (u, \lambda) \rangle + \beta \leq 0$, so $v^*_{j+1} \in \overline{H^-}$. 
     Since $F_i,F_j \subset H^+$, we have that $v^*_i, v^*_j \in H^+$. Since $\gamma^v{^*}$ is degree $2$, and $v_i^*,v_j^* \in H^+$, the arc $\gamma_{i,j}^{v^*}$ intersects $H$ either $0$ or $2$ times. 
    Furthermore, $v_j^* \in H^+$ and $v_{j+1}^* \in \overline{H^-}$ imply that the arc $\gamma_{i,j+1}^{v^*}$ intersects $H$ exactly once. 
     Thus, $\gamma^{v^*}_{i,j} \subset \gamma^{v^*}_{i,j+1}$ implies that $\gamma^{v^*}_{i,j}$ intersects $H$ at most once, and hence not at all.

     This implies that  $\gamma^{v^*}_{ij} \subset H^+$, and hence $v^*_i, v^*_{i+1}, \dots v^*_j \in  H^+$. Equivalently, for all $l = i,\dots,j$, we have $\langle v^*_l,(u,\lambda)\rangle + \beta > 0$.
     Now, recall that $v^* \in \verts(F)$ was chosen so that $\langle v^*,u\rangle \leq \langle v,u \rangle$ for all $v \in \verts(F)$. Thus,
    for each $l=i,\dots,j$ and $v \in \verts(F)$ we have that 
    $$0 \leq \langle v^*_l, (u,\lambda) \rangle + \beta
    = \sqrt{1 - (\tfrac{l}{n})^2} \ \langle v^*,u \rangle + \lambda \tfrac{l}{n} + \beta
    \leq \sqrt{1 - (\tfrac{l}{n})^2} \ \langle v,u \rangle + \lambda \tfrac{l}{n} + \beta = \langle v_l,(u,\lambda)\rangle + \beta
    $$
    and hence $v_l \in H^+$ and $F_{l,l+1} \subset H^+$.

    To show \labelcref{item:globe-intersections-punchline}, assume for contradiction there are three indices $i_1 < i_2 < i_3$ with $\dim(F_{i_j,i_j+1} \cap H)=k$ 
    and $F_{i_j} \cap H = \emptyset $ for all 
    $j \in [3]$.
    Note that for each $j=1,2,3$, $F_{i_j} \cap H = \emptyset$ implies that $F_{i_j}$ is contained in one of the open halfspaces $H^+,H^-$. First, suppose $F_{i_1},F_{i_3}$ lie in the same open halfspace, without loss of generality in $H^+$. Then \labelcref{item:globe-intersections-step1} implies that $F_{i_2,i_2+1} \subset H^+$, contradicting that $\dim(F_{i_2,i_2+1} \cap H) = k$.
    Thus, $F_{i_1}, F_{i_3}$ are contained in distinct halfspaces, so we have $F_{i_1} \subset H^-$ and $F_{i_3}\subset H^+$ (up to reorientation of $H$). 
    If $F_{i_2} \subset H^+$, then applying \labelcref{item:globe-intersections-step1} to $F_{i_2}$ and $F_{i_3}$ gives
    $F_{i_2,i_2+1}\subset H^+$, contradicting
    $\dim(F_{i_2,i_2+1}\cap H)=k$. Hence $F_{i_2}\subset H^-$.
    We therefore have that $F_{i_1},F_{i_2} \subset H^-$. In this case, \labelcref{item:globe-intersections-step1} (after replacing $H^+$ with $H^-$ in the statement) implies that $F_{i_1,i_1+1} \subset H^-$, yielding a contradiction. Thus, the existence of $i_1,i_2,i_3$ implies that $F_{i_1},F_{i_3}$ can neither lie in the same halfspace, nor in different ones. Thus, there exist at most two such indices.
\end{proof}

We continue by proving another technical lemma, where we give an upper bound for the number of $k$-dimensional faces of a polytope $P$ that have a nonempty intersection with a hyperplane $H$. We will use this lemma as well as the prior one in the following proposition to give an upper bound for the f-vector of a slice of $\globe(n,P)$.

\begin{lemma}\label{lem:bound_globeslices-hilfslemma}
    Let $P \subset \R^d$ be a polytope, let $H \subset \R^d$ be an affine hyperplane, and let $0 \leq k \leq d$. Then
    \[
        \#\{F \in \mathcal L(P) \mid \dim(F)=k,\ F \cap H \neq \emptyset\}
        \leq
        3 \max_{\substack{j \in [0,d-1] \\ H' \subset \R^d}} f_j(P \cap H'),
    \]
    where the maximum is taken over all affine hyperplanes $H' \subset \R^d$ such that $ P \cap H' \neq \emptyset$.
\end{lemma}
\begin{proof}
First, all distinct $k$-faces of $P$ contained in $H$ are distinct $k$-faces of $P \cap H$ and so
\begin{align*}
    \#\{F \in \Lcal(P) \mid \dim(F) =k, \ F \subset H \} \leq f_k(P \cap H).
\end{align*}
If $k=0$, then every $0$-face $F$ with $F\cap H\neq\emptyset$ satisfies $F\subset H$, so the claim follows from the previous inequality. We may therefore assume that \(k\geq1\).

    Next, we bound the number of $k$-faces of $P$ intersected by $H$ but not contained in $H$.
    Since $H$ is an affine hyperplane, we have $
        H = \{x \in \R^d \mid \langle u,x\rangle + \beta = 0\}
    $ for some $u \in \R^d, \beta \in \R$.
    For $\varepsilon \in \R$, define
    \[
        H(\varepsilon) = \{x \in \R^d \mid \langle u,x\rangle + \beta = \varepsilon\}.
    \]
    Then $H(\varepsilon)$ and $H(-\varepsilon)$ are parallel to $H$ and contained in opposite halfspaces defined by $H$. Since there are only finitely many faces of $P$, we can choose $\varepsilon$ sufficiently small such that every $k$-face $F$ with $F \cap H \neq \emptyset$ and $F \not\subset H$ is properly intersected by at least one of the hyperplanes $H(\varepsilon)$ and $H(-\varepsilon)$. Distinct $k$-faces properly intersected by either hyperplane define distinct $(k-1)$-faces of the corresponding slice. Thus,
    \[
    \begin{aligned}
        \#\{F \in \mathcal L(P) \mid \dim(F)=k,\ F \cap H \neq \emptyset\}
        &\leq f_k(P \cap H) + f_{k-1}(P \cap H(\varepsilon)) + f_{k-1}(P \cap H(-\varepsilon)) \\
        &\leq 3 \max_{\substack{j \in [0,d-1] \\ H' \subset \R^d}} f_j(P \cap H') \ .
    \end{aligned} \qedhere
    \]
\end{proof}

In the following statement, we bound the number of $k$-dimensional faces of an affine slice of a globe in terms of the base polytope. This is the main ingredient in showing that the globe construction yields a small-slice sequence.

\begin{proposition}
    For a $d$-dimensional polytope $P \subset \R^d$ containing the origin in its interior, an integer $0 \leq k \leq d-1$ and $n \in \N$, we have 
    \[ f_k(\globe(n,P) \cap H) \leq 2f_k(P) + 6(2n-1) \cdot \max_{j,H'}f_j(P\cap H')\ , \]
    where the maximum is taken over all $j \in [0,d-1]$ and all affine hyperplanes $H'$.
    \label{prop:bound_globeslices}
\end{proposition}

\begin{proof}
We first cover the case where $H$ is horizontal, i.e., $H = \{x \in \R^{d+1} \mid x_{d+1} = c\}$ for some $c \in \R$. If $c < \frac{-n+1}{n}$ or $c > \frac{n-1}{n}$, then $\globe(n,P) \cap H = \emptyset$, and the claim holds trivially. Otherwise, there exists some $i \in \{-n+1,\dots,n-2\}$ such that $\tfrac{i}{n} \leq c \leq \frac{i+1}{n}$, so $\globe(n,P) \cap H$ is a slice of the slab $\globe(n,P) \cap \{x \mid \tfrac{i}{n} \leq x_{d+1} \leq \frac{i+1}{n} \}$. By construction of the globe, this slab is a generalized prism over $P$, and hence $\globe(n,P) \cap H$ is combinatorially equivalent to $P$ (see \Cref{prop:generalized-prism-slices} and the proof of \Cref{prop:prism-slices}), and we have $f_k(\globe(n,P) \cap H) = f_k(P)$. Thus, the claim holds for such horizontal hyperplanes $H$.

Let $H$ be a hyperplane which is not horizontal. Recall (by the proof of \Cref{lemma:fkmin1_of_slice_leq_fk_of_poly}) that every $k$-face of $\globe(n,P) \cap H$ can be expressed as the intersection of $H$ with a $(k+1)$-face of $\globe(n,P)$, and that such an assignment from $k$- to $(k+1)$-faces can be made injectively. 
    Thus,
    \begin{align*}
        f_k(\globe(n,P) \cap H) \leq \# \{F \in \mathcal{L}(\globe(n,P)) \mid \dim F = k+1, \ \dim(F \cap H) = k\}
    \end{align*}
    By \Cref{prop:fvector_globe}, if $k \leq d-2$, then there are exactly two types of $(k+1)$-faces of $\globe(n,P)$, namely, the faces of the form $F_{i,i+1}$ (where $F$ is a face of dimension $k$) and those of the form $F_{i}$ (where $F$ is a face of dimension $k+1$). Note that for $i = n-1$, $F_{n-1,n}$ is not defined. If $k = d-1$, then there are again two types of $(k+1)$-faces of $\globe(n,P)$, namely, the faces of the form $F_{i,i+1}$ (where $F$ is a face of dimension $d-1$) and the two poles $P_{n-1},P_{-n+1}$. Note that for $P$ itself, $P_{i}$ for $-n+1 < i < n-1$ are not defined.
    Thus, in either case, with $I = \{-n+1,\dots,n-1\}$ and $J = I \setminus \{n-1\}$ the last expression from above equals
    \begin{align*}
        = \sum_{\substack{F \in \mathcal{L}(P), \\ \dim F = k}} \# \{ i \in J \mid \dim(F_{i,i+1} \cap H) = k \} + \sum_{\substack{F \in \mathcal{L}(P), \\ \dim F = k+1}} \# \{ i \in I \mid \dim(F_i \cap H) = k \} 
    \end{align*}
    where for $k = d-1$ we restrict the sums to those indices where the faces are defined.

    We focus on the first sum.
    By \Cref{lem:globe-intersections} \labelcref{item:globe-intersections-punchline}, for each $k$-dimensional face $F$ of $P$, at most two indices $i \in J$ with $\dim(F_{i,i+1} \cap H) = k$ satisfy $F_i \cap H = \emptyset$. Therefore, we obtain the following comparison to the last equation.

    \begin{align*}
    &\leq \sum_{\substack{F \in \mathcal{L}(P), \\ \dim F = k}} \lp \# \{i \in J \mid F_i \cap H \neq \emptyset \} + 2 \rp + \sum_{\substack{F \in \mathcal{L}(P), \\ \dim F = k+1}} \# \{ i \in I \mid \dim(F_i \cap H) = k \}\\
    &\leq \sum_{\substack{F \in \mathcal{L}(P), \\ \dim F = k}} \lp \# \{i \in I \mid F_i \cap H \neq \emptyset \} + 2 \rp + \sum_{\substack{F \in \mathcal{L}(P), \\ \dim F = k+1}} \# \{ i \in I \mid \dim(F_i \cap H) = k \}\\
    &= 2f_k(P) + \sum_{\substack{F \in \mathcal{L}(P), \\ \dim F = k}} \# \{ i \in I \mid F_i \cap H \neq \emptyset \} + \sum_{\substack{F \in \mathcal{L}(P), \\ \dim F = k+1}} \# \{ i \in I \mid \dim(F_i \cap H) = k \}
    \end{align*}
    Changing the order of summation yields
    \begin{equation}
    \begin{aligned}\label{eq:long-proof}
    &= 2f_k(P) + \sum_{i \in I} \# \{F_i \in \mathcal L(\globe(n,P)) \mid \dim(F_i) = k, F_i \cap H \neq \emptyset \} 
    \\ 
    & \hspace{4.3em} 
    + \sum_{i \in I} \# \{F_i \in \mathcal L(\globe(n,P)) \mid \dim(F_i) = k+1, \dim(F_i \cap H) = k \} \ .
    \end{aligned}
    \end{equation}

    We proceed by bounding each of the sums separately. For $i \in I$, define 
    \[
        S_i = \globe(n,P) \cap \{x \in \R^{d+1} \mid x_{d+1} = \frac{i}{n} \}
    \]
    Then $S_i$ is affinely isomorphic to $P$ via the affine map $T_i(x) = ( \sqrt{1 - (\tfrac{i}{n})^2} x , \tfrac{i}{n})$. Note that for any $F \in \mathcal L(P)$ we have $F_i \preceq S_i$. If $H \subset \R^{d+1}$ has nonempty intersection with $F_i$, then it has nonempty intersection with $S_i$. Since $H$ is not horizontal, we have that $H$ does not contain $S_i$, so we can define $H_i = T_i^{-1}(H \cap \{x \mid x_{d+1} = \tfrac{i}{n}\})$ to be the corresponding hyperplane in $\R^d$. By construction, we have that $F_i \cap H \neq \emptyset$ if and only if $F \cap H_i \neq \emptyset$, and $\dim(F_i \cap H) = \dim(F \cap H_i)$. Thus, for fixed $i \in I$, and $k \in \{0,\dots,d-1\}$, we have
    \begin{align*}
        \# \{F_i \in \mathcal L(\globe(n,P)) \mid \dim(F_i) = k, F_i \cap H \neq \emptyset \} &= \# \{F \in \mathcal L(P) \mid \dim(F) = k, F \cap H_i \neq \emptyset \} \ .
    \end{align*}
    By \Cref{lem:bound_globeslices-hilfslemma}, this is bounded from above by $3 \max_{j,H'} f_j(P \cap H')$, which bounds the first sum of \eqref{eq:long-proof}. 
    In the second sum, for fixed $i \in I$ and $k \in \{0,\dots,d-1\}$, we have
    \begin{align*}
     &\# \{F_i \in  \mathcal L(\globe(n,P)) \mid \dim(F_i) = k+1, \dim(F_i \cap H) = k \} \\
    &\leq \# \{F \in \mathcal L(P) \mid \dim(F) = k+1, \dim(F \cap H_i) =k \} \\
    &\leq \# \{F \in \mathcal L(P) \mid \dim(F) = k+1, F \cap H_i \neq \emptyset \} \ .
    \end{align*}
    Again, \Cref{lem:bound_globeslices-hilfslemma} yields an upper bound by $3 \max_{j,H'} f_j(P \cap H')$. 
    Altogether, we thus obtain
    \begin{align*}
        f_k(\globe(n,P) \cap H) 
        &\leq
        2f_k(P) + \sum_{i \in I} \# \{F_i \in \mathcal L \mid \dim(F_i) = k, F_i \cap H \neq \emptyset \} 
    \\ 
    & \hspace{4.3em} 
    + \sum_{i \in I} \# \{F_i \in \mathcal L \mid \dim(F_i) = k+1, \dim(F_i \cap H) = k \} \\
    &\leq 2 f_k(P) + |I| \Big( 3 \max_{\substack{j \in [0,d-1] \\ H' \subset \R^d}} f_j(P \cap H') \Big) + |I| \Big( 3 \max_{\substack{j \in [0,d-1] \\ H' \subset \R^d}} f_j(P \cap H') \Big)  \\
    &\leq 2 f_k(P) + 6(2n-1) \max_{\substack{j \in [0,d-1] \\ H' \subset \R^d}} f_j(P \cap H') \ . \qedhere
    \end{align*}
\end{proof}

We are now ready to show that the globe over a small-slice sequence in dimension $d$ yields a small-slice sequence in dimension $d+1$.

\begin{proposition}
    If $(P_n)_{n \in \N}$ is a small-slice sequence of $d$-dimensional polytopes containing the origin in their relative interiors, then $(\globe(n,P_n))_{n \in \N}$ is a small-slice sequence of $(d+1)$-dimensional polytopes.
    \label{prop:smallslice_induction}
\end{proposition}

\begin{proof}
    By assumption, $f_k(P_n) \in \Theta(n^{d-1})$ for all $k=0,\dots,d-1$. 
    From the description of the $f$-vector of $\globe(n,P_n)$ in  \Cref{prop:fvector_globe}, we obtain $f_k(\globe(n,P_n)) \in \Theta(n^d)$.
    Moreover, as $(P_n)_{n \in \N}$ is a small-slice sequence, we also have $\max_{k,H}f_k(P_n \cap H) \in \mathcal{O}(n^{d-2})$.  Thus, \Cref{prop:bound_globeslices} implies that 
    \[ f_k(\globe(n,P_n) \cap H) \leq 2f_k(P_n) + 6(2n-1) \max_{k,H}f_k(P_n\cap H) \in \mathcal{O}(n^{d-1}),  \]
    so $(\globe(n,P_n))_{n \in \N}$ is a small-slice sequence.
\end{proof}

In light of the previous \namecref{prop:smallslice_induction}, we are left with finding a small-slice sequence in low dimension. Repeatedly applying the $\globe$ construction to each element of the sequence we obtain a small-slice sequence in general dimension. Any sequence of regular polygons with $m \in \N$ vertices suffices as such a sequence. Recall from \Cref{def:globes} that a $d$-dimensional globe $G^{d}(m)$ is the $(d-2)$-fold globe over a regular $m$-gon.

\begin{corollary}\label{cor:small-slice-iterated-globes}
    For any $d \geq 2$,  $(G^{d}(m))_{m \in \N}$
    is a small-slice sequence of $d$-dimensional polytopes.
\end{corollary}

\begin{proof}
    Let $(P_m)_{m \in \N}$ be a sequence of regular $m$-gons in $\R^2$.
    We observe that $(P_m)_{m \in \N}$ is a small-slice sequence, as $f_0(P_m) = f_1(P_m) = m \in \Theta(m^{2-1})$. For any hyperplane $H$, $P_m \cap H$ has at most two vertices and one edge, so $f_0(P_m \cap H), f_1(P_m \cap H) \in \mathcal{O}(1) = \mathcal{O}(m^{2-2})$. By \Cref{prop:smallslice_induction}, globes over small-slice sequences are small-slice sequences, which implies the statement.
\end{proof}

\subsection{Big-slice sequences and bipyramids}

Having constructed a $d$-dimensional small-slice sequence in \Cref{sec:small-slice}, it is now time to construct a big-slice sequence. We begin by showing that, given a big-slice sequence in dimension $d$, taking bipyramids yields a big-slice sequence in dimension $d+1$.

\begin{proposition}\label{prop:big-slice-pyramids}
    Let $(P_n)_{n \in \N}$ be a big-slice sequence of $d$-dimensional polytopes, containing the origin in their relative interiors. Then $(\bipyr(P_n))_{n \in \N}$ is a big-slice sequence of $(d+1)$-dimensional polytopes.
\end{proposition}

\begin{proof}
    We know that 
    \[f_k(\bipyr(P_n)) = \begin{cases}
    f_k(P_n) + 2f_{k-1}(P_n) & \text{for } 0 \leq k \leq d-1 \\
    2f_{k-1}(P_n) & \text{for } k=d \ . 
    \end{cases}\] 
    Since by assumption $f_k(P_n) \in \Theta(n)$ for $0 \leq k \leq d-1$, it follows that $f_k(\bipyr(P_n)) \in \Theta(n)$ for all $0 \leq k \leq d$.
    
    By \Cref{prop:slices_bip} we know that $f_k(\bipyr(P_n) \cap u^\perp)  \geq \min(f_k(P_n), s_k+2s_{k-1})$ for all $u \in \Sph^{d}$, where $s_k:= \min _{u \in \Sph^{d-1}}f_k(P_n \cap u^\perp)$ for all $0 \leq k \leq d-2$, and $s_{-1} = 1, s_{d-1} = 0$. Since $(P_n)_{n \in \N}$ is a big-slice sequence, we have $s_k \in \Omega(n)$ for $0 \leq k \leq d-2$, so $s_k+2s_{k-1} \in \Omega(n)$ for $0 \leq k \leq d-1$, and $f_k(P_n) \in \Theta(n)$ for $0 \leq k \leq d-1$.
    Therefore $\min_{u \in \Sph^d} f_k(\bipyr(P_n) \cap u^\perp) \in \Omega(n)$ for $0 \leq k \leq d-1$. \\
\end{proof}

It remains to construct a $3$-dimensional big-slice sequence.

\begin{definition}[Shell polytope]\label{def:shell}
    For $n \geq 3$ and $1 \leq i \leq n-1$, let
    \begin{equation*}
        \begin{array}{rl}
            \theta_i := -\frac{\pi}{4} + (i-1)\frac{\pi}{2(n-2)}, \qquad
            v_i := (\sin\theta_i, 0, \cos\theta_i), \qquad 
            w_i := (0, \sin\theta_i, -\cos\theta_i) \ .
        \end{array}
    \end{equation*}
    We define the \emph{n-shell polytope} as
    \begin{equation*}
        S_n := \conv(v_1, \ldots, v_{n-1}, w_1, \ldots, w_{n-1}) \subset \R^3 \ .
    \end{equation*}
\end{definition}

As can be seen in \Cref{fig:shell14} all vertices $v_i$ and $w_i$ lie on two perpendicular arcs. Concretely, the $v_i$ lie in the circle $x^2+z^2=1$, $y=0$ and the $w_i$ lie in the circle $y^2+z^2=1$, $x=0$.  We continue by formally describing the facial structure of a shell polytope.

\begin{figure}[h]  
    \centering
    \includegraphics[width=0.2\linewidth]{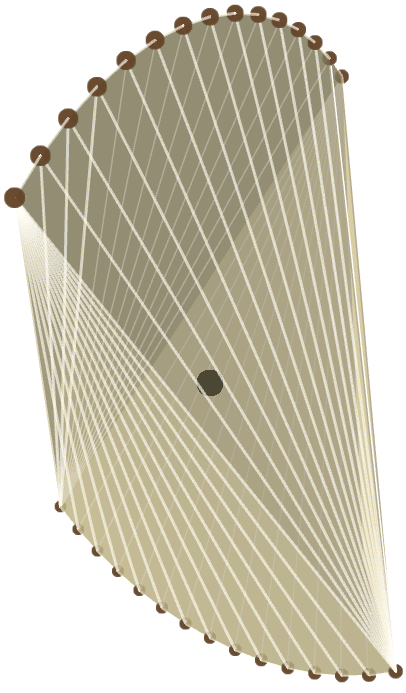}
    \caption{The $14$-shell polytope.}
    \label{fig:shell14}
\end{figure}

\begin{lemma}\label{lem:shell-faces}
    The $n$-shell polytope is a simplicial $3$-dimensional polytope, it contains the origin in its interior, and its facets are
    \[
        \conv(v_i,v_{i+1},w_1), \quad
        \conv(v_i,v_{i+1},w_{n-1}), \quad
        \conv(w_i,w_{i+1},v_1), \quad
        \conv(w_i,w_{i+1},v_{n-1}), 
    \]
    for $i \in [n-2]$. In particular, $f_1(S_n) = 6n-12$, and $f_2(S_n)=4n-8$.
\end{lemma}
\begin{proof}
    First note that $\conv(v_1,v_{n-1},w_1,w_{n-1})$ is a $3$-dimensional simplex containing the origin in its interior. Since this is a subpolytope of $S_n$, the same holds for $S_n$ itself.
    We begin by excluding faces of types other than those listed in the statement. To show that there is no face containing $v_i,v_j,v_k$, note that $v_1,\dots,v_{n-1}$ lie on a convex arc. Hence, every plane intersects this arc at most twice, or it contains the entire arc. The plane containing the entire arc is the plane $\{x \in \R^3 \mid x_2 = 0\}$, which separates $w_1$ and $w_{n-1}$ into opposite halfspaces, hence it is not a supporting hyperplane. The same argument excludes faces with vertices $w_i,w_j,w_k$. 
    Consider now a hyperplane $H$ through $v_i,v_j,w_k$ for $i < j$ and some $k \in [n-1]$.  Then any $x \in H$ satisfies the equation
    \begin{multline*}
\sin\theta_k \sin\left(\frac{\theta_i+\theta_j}{2}\right)x_1
+
\left(
\cos\left(\frac{\theta_j-\theta_i}{2}\right)
+
\cos\theta_k\cos\left(\frac{\theta_i+\theta_j}{2}\right)
\right)x_2
+
\sin\theta_k \cos\left(\frac{\theta_i+\theta_j}{2}\right)x_3 \\
=
\sin\theta_k\cos\left(\frac{\theta_j-\theta_i}{2}\right).
\end{multline*}
    If $w_k \not \in \{w_1,w_{n-1}\}$, then $w_1,w_{n-1}$ are on opposite sides of the corresponding hyperplane. If $w_k \in \{w_1,w_{n-1}\}$ and $ v_j \neq v_{i+1}$, then the hyperplane separates $v_{i+1}$ from $v_{1},\dots,v_{i-1},v_{j+1},\dots,v_{n-1}$.
    Otherwise, i.e., if $\{v_i,v_j,w_k\} \in \{ \{v_i,v_{i+1},w_1\},\{v_i,v_{i+1},w_{n-1}\} \}$, then this expression is maximized or minimized exactly on the three vertices. The faces of the form $w_i,w_j,v_k$ follow analogously. Thus, the shell polytope has $2(n-1)$ vertices and $f_2(S_n) = 4(n-2) = 4n - 8$ facets. By Euler's formula, the number of edges is therefore $f_1(S_n) = f_0(S_n) + f_2(S_n) - 2 = 6n - 12$.
\end{proof}

It remains to show that a sequence of $n$-shell polytopes is indeed a big-slice sequence in dimension $3$. By \Cref{lem:shell-faces} we already have $f_k(S_n) \in \Theta(n)$ for $k=0,1,2$. Since central slices of the $n$-shell are two-dimensional we have $f_0(S_n \cap u^\perp) = f_1(S_n \cap u^\perp)$, and so it remains to show that
$\min_{u \in \Sph^2} f_0(S_n \cap u^\perp) \in \Omega(n)$.

\begin{lemma}
   For the $n$-shell polytope $S_n$ and $u \in \Sph^2$ it holds that $f_{0}(S_n \cap u^\bot) \geq n$.
\label{lem:shellboundslices}
\end{lemma}

\begin{proof}
    Let $u=(a,b,c)$ and write $H=u^\perp$. 
    Let $\gamma$ be the arc in the circle $\{(x,y,z) \in \R^3 \mid x^2+z^2=1,y=0\}$, from $v_1$ to $v_{n-1}$, on which all $v_i$ lie. If $H$ intersects $\gamma$ in two distinct points, then together with the origin, $H$ contains three points of $xz$ plane and hence $H$ contains the entire arc. We distinguish between three cases: $H$ contains $\gamma$, $H$ does not intersect $\gamma$ or $H$ intersects $\gamma$ in one point.

    If $H$ contains $\gamma$, then $u = \pm (0,1,0)$. In this case, $H$ contains $v_1, \ldots, v_{n-1}$, and intersects the edge $\conv(w_{\lfloor\frac{n}{2}\rfloor}, w_{\lfloor\frac{n}{2}\rfloor+1})$, and hence $f_0(S_n \cap u^\perp) \geq n$.
    If $H$ does not intersect $\gamma$, then all $v_i$ are in one open halfspace of $H$,
    say $H^+$. So $\langle v_i,u \rangle > 0$ for all $i \in [n-1]$, and 
    \[
        \langle v_1,u \rangle = \frac{\sqrt{2}}{2} (-a + c) > 0\ , \qquad \langle v_{n-1},u \rangle = \frac{\sqrt{2}}{2} (a + c) > 0 \ ,
    \]
    and hence $c > |a|$. In particular, $c > 0$. Moreover, we have
    \[
        \langle w_1,u \rangle = \frac{\sqrt{2}}{2} (-b - c) \ , \qquad \langle w_{n-1},u \rangle = \frac{\sqrt{2}}{2} (b - c) \ ,
    \]
    and $c > 0$ implies that at least one of these inner products is negative. Thus, at least one of the vertices $w_1,w_{n-1}$ is contained in $H^-$.
    If both $w_1, w_{n-1} \in H^-$, then $u^\perp$ intersects all edges of the form $v_i w_k$ for $k \in \{1,n-1\}$ and all $i \in [n-1]$, so $f_0(S_n \cap u^\perp) = 2(n-1)$. Otherwise, suppose without loss of generality that $w_1 \in H^-, w_{n-1} \in \overline{H^+}$. Then there exists some $j \in [n-2]$ such that $w_1,\dots,w_j \in H^-, w_{j+1} \in \overline{H^+}$ and $w_{j+2},\dots,w_{n-1} \in H^+$. Thus, $H$ intersects the edge $w_jw_{j+1}$, and all edges $v_iw_1$, for $i \in [n-1]$,
    yielding $f_0(S_n \cap u^\perp) \geq n $.
    
    Finally, suppose that $H$ intersects $\gamma$ once. Then w.l.o.g. $\langle v_1, u \rangle \geq 0$ and $\langle v_{n-1}, u \rangle < 0$ and there exists some edge $v_j v_{j+1}$ for $j \in [n-2]$ which has nonempty intersection with $u^\perp$. 
    If $\langle v_1,u\rangle=0$, then $v_2,\ldots,v_{n-1}\in H^-$. Moreover, at least one of $w_1,w_{n-1}$ lies in $H^+$, since $H = u^\perp$ is not a supporting hyperplane. Suppose without loss of generality that $w_{n-1}\in H^+$ (otherwise, interchange the roles of $w_1$ and $w_{n-1}$ below). Then $H$ properly intersects the $n-2$ edges $v_iw_{n-1}$ for $i=2,\ldots,n-1$. Together with the vertex $v_1\in H$, this gives $n-1$ vertices of the slice. If $w_1\in H^+$, then the edge $v_2w_1$ gives one further vertex. If $w_1\in H$, then $w_1$ itself gives one further vertex. If $w_1\in H^-$, then $H$ intersects some edge $w_jw_{j+1}$. Hence, in all cases, $f_0(S_n\cap u^\perp)\geq n$. 
    We are left with the case where $\langle v_1,u\rangle>0$ and $\langle v_{n-1}, u \rangle < 0$.
    Let $i \in [n-1]$. If $\langle w_i, u \rangle = 0$, then $w_i$ is a vertex of $S_n \cap u^\perp$. If $\langle w_i, u \rangle > 0$, then $u^\perp$ intersects the edge $v_{n-1}w_i$ in its relative interior. If $\langle w_i, u \rangle < 0$, then $u^\perp$ intersects the edge $v_{1}w_i$ in its relative interior. 
    Thus, $f_0(S_n \cap u^\perp) \geq n$.
\end{proof}

\begin{proposition}\label{prop:shell-big-slice}
    $(S_n)_{\substack{n \geq 3}}$ is a big-slice sequence of $3$-dimensional polytopes.
\end{proposition}

\begin{proof}
    To be a big-slice sequence we need to show that $f_k(S_n) \in \Theta(n)$ for $k=0,1,2$ and that $\min_{u \in \Sph^2}f_k(S_n \cap u^\perp) \in \Omega(n)$ for $k=0,1$.
    By definition $f_0(S_n) = 2n-2 \in \Theta(n)$, and by \Cref{lem:shell-faces} $f_1(S_n) = 6n-12 \in \Theta(n)$ and
    $f_2(S_n)=4n-8\in \Theta(n)$.
    On the other hand, $f_0(S_n \cap u^\perp)=f_1(S_n \cap u^\perp)$ and \Cref{lem:shellboundslices},
    $f_1(S_n\cap u^\perp)\geq n$ for every $u \in \Sph^2$. Hence
    $\min_{u\in\Sph^2} f_1(S_n\cap u^\perp) \in \Omega(n)$.
\end{proof}

\begin{corollary}\label{cor:big-slice-bipyramid-shell}
    For any $d \geq 3$,  iterated bipyramids $(\bipyr^{d-3}(S_n))_{n \geq 3}$ over a sequence of $n$-shell polytopes form a big-slice sequence of $d$-dimensional polytopes.
\end{corollary}
\begin{proof}
    By \Cref{prop:shell-big-slice}, $(S_n)_{n \geq 3}$ is a big-slice sequence of $3$-dimensional polytopes. By \Cref{prop:big-slice-pyramids}, bipyramids over big-slice sequences are big-slice sequences, which implies the statement.
\end{proof}

Combining the small-slice sequence constructed in \Cref{sec:small-slice} with the big-slice sequence from above yields the desired counterexamples to the combinatorial versions of the Busemann-Petty and Bourgain's slicing problem.

\begin{theorem}\label{th:disprove-BP}
    The combinatorial versions of the Busemann-Petty problem and Bourgain's slicing problem are false. 
    More precisely, for every dimension $d \geq 3$, every $k \in [d-1]$, and for every constant $C > 0$,
    there exist $n,m \in \N$ such that $f_{k-1}(G^{d}(m) \cap u^\perp) \leq f_{k-1}(\bipyr^{d-3}(S_n) \cap u^\perp)$ for every $u \in \Sph^{d-1}$, and $f_k(G^{d}(m)) > C \cdot f_k(\bipyr^{d-3}(S_n))$.
\end{theorem}
\begin{proof}
    By \Cref{prop:bp_counterexample}, every big-slice sequence and small-slice sequence of $d$-dimensional polytopes will yield a counterexample. 
    By \Cref{cor:big-slice-bipyramid-shell}, $(\bipyr^{d-3}(S_n))_{n \geq 3}$ is a big-slice sequence, and by \Cref{cor:small-slice-iterated-globes} $(G^{d}(m))_{m\in \N}$ is a small-slice sequence.
\end{proof}

\begin{remark}[Affine versions]\label{rem:bpp-affine}
    We note that the counterexample constructed in this section can also be adapted to counterexamples to affine versions of the combinatorial Busemann-Petty and Bourgain's slicing problem, in which we replace central with affine hyperplanes. The bound for the small-slice sequence already holds for affine slices in its current formulation (\Cref{prop:bp_counterexample}). We now describe how to obtain a valid bound for the big-slice sequence for affine slices.
    Note that for the shell-polytope, there exist affine hyperplanes that produce small slices: For example, cutting off a single vertex of a $3$-dimensional shell may yield a quadrilateral slice. However, to produce such a small slice, the hyperplane needs to have a certain distance from the origin. Indeed, one can show that there exists some radius $\varepsilon > 0$ such that all affine hyperplanes which have nonempty intersection with a ball $B_\varepsilon(\zero)$ exhibit all properties of a big-slice sequence, and that bipyramids preserve these properties (although the radius $\varepsilon > 0$ may change).
    We can then easily construct a counterexample for affine hyperplanes by taking a counterexample for central hyperplanes and scaling the small-slice sequence of polytopes to be contained in the ball $B_\varepsilon(\zero)$. Admittedly, this does not seem like a very satisfying solution, so it might be worthwhile to investigate affine variants of the combinatorial slicing problems which prevent such a ``simple'' solution.
\end{remark}

\begin{remark}[Centrally symmetric versions]\label{rmk:BP-symmetric}
    The original, metric formulations of the Busemann-Petty and Bourgain's slicing problem restrict to centrally symmetric convex bodies. Our combinatorial versions do not ask for this requirement. While we do not see a reason why this requirement should be relevant for the combinatorial versions, we are currently unaware of a centrally symmetric counterexample, thus leaving this case open. Note that the globe, i.e., the small-slice sequence constructed in this section can already be made centrally symmetric (by choosing a sequence of regular $2n$-gons for $n \in \N$). However, the shell-polytope is far from being centrally symmetric.
\end{remark}

    \begin{question}
        Do the combinatorial Busemann-Petty and Bourgain's slicing problems hold for centrally symmetric polytopes? In particular,
        is there a centrally symmetric big-slice sequence?
    \end{question}

\printbibliography

\vfill

\vfill
\begin{center}
\begin{minipage}{.42\textwidth}
\noindent
\textsc{Anna Birkemeyer}\\
\textsc{Universität Osnabrück}\\
\url{abirkemeyer@uos.de} \\

\noindent
\textsc{Anouk Brose}\\
\textsc{University of California, Davis}\\
\url{aebrose@ucdavis.edu}
\end{minipage}
\begin{minipage}{.42\textwidth}
\noindent
\textsc{Marie-Charlotte Brandenburg}\\
\textsc{Ruhr-Universit\"at Bochum}\\
\url{marie-charlotte.brandenburg@rub.de}\\

\noindent
\textsc{Niklas Prün}\\
\textsc{Karlsruhe Institute of Technology}\\
\url{niklas.pruen@student.kit.edu}
\end{minipage}
\end{center}





\end{document}